\documentclass{amsart}
\usepackage{amssymb}
\usepackage[active]{srcltx}
\usepackage{a4wide}
\usepackage{amsmath,amsthm}
\usepackage{amssymb}
\usepackage{amstext}
\usepackage{cite}
\usepackage{epsfig}
\usepackage{enumerate}
\newtheorem{theorem}{Theorem}[section]
\newtheorem{lemma}[theorem]{Lemma}

\theoremstyle{definition}
\newtheorem{definition}[theorem]{Definition}
\newtheorem{example}[theorem]{Example}

\theoremstyle{remark}
\newtheorem{remark}[theorem]{Remark}

\numberwithin{equation}{section}

\begin{document}

\title{A Fractalization of  Rational Trigonometric Functions}


\author{S. Verma}
\email{saurabh331146@gmail.com; viswa@maths.iitd.ac.in}
\author{P. Viswanathan}
\address{Department of Mathematics, IIT Delhi, New Delhi, India 110016}





\begin{abstract}

In \cite{MAN1, Rational}, new approximation classes of self-referential functions are introduced as fractal versions  of the classes of polynomials and rational functions. As a sequel, in the present article, we define a new approximation class consisting of self-referential functions, referred to as the fractal rational trigonometric functions. We establish Weierstrass type approximation theorems for this class and prove the existence of a best fractal rational trigonometric approximant to a real-valued continuous function on a compact interval. Furthermore, we provide  an upper bound for the smallest error in approximating a prescribed continuous function by a fractal rational trigonometric function. This extemporizes an analogous result in the context of fractal rational function appeared in \cite{Rational} and followed in the setting of Bernstein fractal rational functions in \cite{Vij2}. The last part of the article aims to clarify and correct the mathematical errors in some results on the Bernstein $\alpha$-fractal functions  appeared recently in the literature \cite{Vij1,Vij2,Vij3}.

\end{abstract}

\maketitle

\section{Introduction and Preliminaries}
 The notion of fractal interpolation function (FIF in what follows) has proved to be an attractive strategy to produce interpolants and approximants for a wide class of problems. The basic setting of FIF as defined by Barnsley \cite{MF1} stems from the concept of iterated function system (IFS), one of the most popular methods of generating fractals; see, for instance, \cite{Hut}. The book \cite{PRM2} and the monograph  \cite{PRM1} are good references for fractal functions and related areas. The theory of fractal interpolation is an active research topic in the field of fractal approximation theory, as shown, for example in \cite{BL,CK,DCL,MAN2,Ri,S, WY}.  In what follows, we shall hint at the technical details concerning the notion of FIF; the readers are referred to \cite{MF1} for more details.
 \par
Let  $\mathbb{N}_r$ denote  the set of first $r$ natural numbers. As is customary, we shall denote by $\mathcal{C}(I)$ the Banach space of all real-valued continuous
functions on a closed bounded  interval $I$, endowed with the supremum norm.
\par
Let $\{(x_i,y_i): i=0,1,2,...,N\}$,  $N \ge 2$ be a data set with strictly increasing abscissae. Set $I=[x_0,x_N]$ and $I_i= [x_{i-1},x_i]$ for $i \in \mathbb{N}_N$. For every $i \in \mathbb{N}_N$, suppose that $L_i: I \to I_i$ is a contractive increasing homeomorphism and $F_i: I \times \mathbb{R} \to \mathbb{R} $ is a map that is continuous, contractive with respect to the second variable and such that
$$F_i(x_0,y_0) = y_{i-1}, \quad F_i(x_N,y_N)= y_i.$$
For $i \in \mathbb{N}_N$, define
$$W_i(x,y) =  \Big(L_i(x),F_i(x,y) \Big) ~~\forall ~~(x,y) \in I \times \mathbb{R}.$$
The system $\{I \times \mathbb{R}: W_i, i \in \mathbb{N}_N\}$ is an IFS and it has a unique attractor which is the graph of a continuous function  $g: I \to \mathbb{R}$ such that  $g(x_i)=y_i$ for every $i=0,1,\dots, N$ and
$$g(x) = F_i \Big (L_i^{-1}(x),g\big(L_i^{-1}(x)\big)  \Big) ~~\forall ~~x \in I_i.$$
The function $g$ is called a FIF \cite{MF1} and it has the property that its graph is self-referential, that is, the graph is a union of transformed copies of itself. The main
difference with the classical interpolants resides in the definition by the aforementioned functional equation endowing self-referentiality to the interpolant $g$.
\par
Navascu\'{e}s explored the idea of fractal interpolation further to associate a class of fractal functions with a prescribed function $f$ in $\mathcal{C}(I)$ as follows \cite{MAN1}. Consider a partition $\Delta:=\{x_0,x_1, \dots,x_N\}$ of $I=[x_0,x_N]$ such that  $x_0<x_1<\dots< x_N$. For $i \in \mathbb{N}_N$,  let
$$L_i(x)=a_ix+b_i, \quad F_i(x,y) = \alpha_iy + f \big( L_i(x) \big)-\alpha_i b(x),$$
where $b \neq f $ is such that
$$b(x_0) = f(x_0), \quad b(x_N) = f(x_N),$$ and $\alpha:= (\alpha_1, \alpha_2, \dots, \alpha_N) \in (-1,1)^N$. The corresponding FIF denoted by $f_{\Delta,b}^\alpha$ or simply as  $f^\alpha$ (for notational convenience) is called an \emph{$\alpha$-fractal function} and it satisfies the self-referential equation
$$f_{\Delta,b}^\alpha(x) = f(x) + \alpha_i (f_{\Delta,b}^{\alpha}- b)\big(L_i^{-1}(x)\big) ~ ~ ~~\forall~~ x \in I_i,~~ i \in \mathbb{N}_N.$$
The function, or rather the class of functions, $f_{\Delta,b}^\alpha$ may be treated as  fractal perturbation of the original function $f$, termed the \emph{germ function} or \emph{seed function}. Note that the perturbation process involves three elements: the \emph{partition} $\Delta$ of the domain $I$,  function $b$ that is referred to as the \emph{base function}  and vectorial parameter $\alpha$ termed \emph{scaling vector}.  Taking advantage of the scaling vector, fractal interpolation  is more robust than the classical piecewise interpolation.
\par
As was observed by Navascu\'{e}s \cite{MAN2}, a particular interesting case arises if one chooses the base function $b=Lf$, where  $L: \mathcal{C}(I) \to \mathcal{C}(I)$ is a bounded linear map. With fixed choices of $\Delta$, $\alpha$ and $L$, one can define an operator $\mathcal{F}_{\Delta,L}^\alpha$ denoted for simplicity as $\mathcal{F}^\alpha$ that assigns $f_{\Delta,L}^\alpha$ to $f$:
\begin{equation}\label{eq0}
\mathcal{F}_{\Delta,L}^\alpha: \mathcal{C}(I) \to \mathcal{C}(I), \quad \mathcal{F}_{\Delta,L}^\alpha(f)= f_{\Delta,L}^\alpha,
\end{equation}
referred to as the \emph{$\alpha$-fractal operator.} Let $$|\alpha|_\infty:= \max \big\{ |\alpha_i|: i \in \mathbb{N}_N\big\}.$$ The following properties of the $\alpha$-fractal operator $\mathcal{F}_{\Delta,L}^\alpha$ are well-known; see, for instance, \cite{MAN2,Rational}.

\begin{theorem}\textup{\cite[Theorem 2.2]{Rational}}.   \label{prelthm} Let $Id$ be the identity operator on $\mathcal{C}(I)$.
\begin{enumerate}
\item The fractal operator $\mathcal{F}_{\Delta,L}^{\alpha}: \mathcal{C}(I) \to \mathcal{C}(I)$ is a  bounded linear map. Further, the operator norms satisfy the following inequalities $$ \|Id-\mathcal{F}_{\Delta,L}^{\alpha}\| \leq \frac{| \alpha|_{\infty} }{1-| \alpha|_{\infty}} \| Id - L\|,\quad  \| \mathcal{F}_{\Delta,L}^{\alpha}\| \leq 1+ \frac{| \alpha|_{\infty} }{1-| \alpha|_{\infty}} \| Id - L\|.$$
 \item  For $| \alpha |_{\infty} < \|L\|^{-1}$, $\mathcal{F}_{\Delta,L}^{\alpha}$ is bounded below. In particular, $\mathcal{F}_{\Delta,L}^{\alpha}$ is an injective map.

   \item For $| \alpha|_{\infty} < \|L\|^{-1}$, the fractal operator  $\mathcal{F}_{\Delta,L}^{\alpha}$ is not a compact operator.
  \item  If $| \alpha|_{\infty} < \big(1 + \|Id - L\|\big)^{-1}$, then $\mathcal{F}_{\Delta,L}^{\alpha}$ is a topological isomorphism (i.e., a bijective bounded linear map with a bounded inverse). Moreover, $$\| (\mathcal{F}_{\Delta,L}^{\alpha})^{-1}\| \leq \dfrac{1+|\alpha|_{\infty} }{1-|\alpha|_{\infty} \|L\|}.$$

 \end{enumerate}
  \end{theorem}
\par Navascu\'{e}s and coworkers approached the construction of new classes of functions in $\mathcal{C}(I)$ by taking the image of the popular approximation classes of functions such as polynomials, trigonometric and rational functions under the fractal operator $\mathcal{F}_{\Delta,L}^\alpha$ \cite{MAN1,MAN3,Rational}. These new functions defined as perturbations of the classical may preserve properties of the latter or display new characteristics such as non-smoothness or quasi-random behavior. These fractal maps tends to bridge the gap between the smoothness of the classical mathematical objects and the pseudo-randomness of the experimental data, breaking in this way their apparent contradiction \cite{MAN4}. Motivated and influenced by the aforementioned works, in this article we define the class of \emph{fractal rational trigonometric functions} and study some approximation aspects of the same. In this way, we approach the classical problems of periodicity and approximation from a fractal viewpoint. Besides providing the motivation for our researches reported herein, the works in \cite{MAN1,MAN3,Rational} also offered us an
array of basic tools which we have modified and adapted.
\par In Theorem \ref{error1a} we provide  an upper bound for the fractal rational trigonometric minimax error, that is, the smallest error in approximating a prescribed continuous function by a fractal rational trigonometric function in the uniform norm. Our approach to the fractal minimax error also points out that the upper bound for the minimax error in approximating a continuous function by a fractal rational function as announced in \textup{\cite[Theorem 4.3]{Rational}}
does not hold. The second author regrets to inform that this error may invalidate some results  given as corollaries of \textup{\cite[Theorem 4.3]{Rational}}. However, as stated in the analogous  theorem in this paper (see, Theorem \ref{error1a}), there is a way to fix this mistake by including an additional term.  The incorrect arguments in \textup{\cite[Theorem 4.3]{Rational}} is subsequently carried over almost verbatim in \textup{\cite[Theorem 3.9]{Vij2}}. However, we note that the result stated in  \textup{\cite[Theorem 3.9]{Vij2}} remains valid and provide a correct proof for it.
\par
 The following upper bound for the uniform distance between the original function $f$ and its fractal version $f_{\Delta,b}^\alpha$ can be obtained; see, for instance, \cite{MAN1,MAN2}:
\begin{equation} \label{eq1}
\|f_{\Delta,b}^\alpha-f\|_\infty \le \frac{|\alpha|_\infty}{1-|\alpha|_\infty} \|f-b\|_\infty.
\end{equation}
The previous estimate  reveals that by choosing the scaling vector $\alpha$ such that $|\alpha|_\infty$ is close enough  to zero or by selecting the base function $b$ near to $f$, the perturbed fractal function $f_{\Delta,b}^\alpha$ can be made sufficiently  close to the original seed function $f$. In particular, if $\alpha^m \in \mathbb{R}^N$, $|\alpha^m|_\infty<1$ and $\alpha^m \to 0$ as $m \to \infty$, then $f_{\Delta,b}^{\alpha^m} \to f$ uniformly as $m \to \infty$.
\par Since FIFs do not possess a closed form expression, standard methods such as the Taylor
series analysis, Cauchy remainder form, and  Peano kernel theorem (see, for instance, \cite{PJD}) may not be easily
adapted for the convergence analysis of fractal interpolants and approximants. Instead, in the literature, the closeness of a fractal approximant $f_{\Delta,b}^\alpha$ (which is perturbation of a classical approximant $f$) to the original function $\Phi$ is established  using the closeness of $f$ to $\Phi$ via the following  triangle inequality:
\begin{equation} \label{eq2}
\|\Phi -f_{\Delta,b}^{\alpha}\|_\infty \le \|\Phi-f\|_\infty  + \|f_{\Delta,b}^{\alpha}- f\|_\infty.
\end{equation}
The second term in the right hand side of the Inequality (\ref{eq2}) can be bounded via (\ref{eq1}) to conclude that for the scaling vector $\alpha$ with small enough value of $|\alpha|_\infty$, the error in approximating  $\Phi$ with $f_{\Delta,b}^\alpha$ is small, whenever $f$ is a good approximant to $\Phi$. As various fractal interpolants studied in the literature can be realized as fractal perturbation of their classical counterparts, a similar comment holds for their convergence (see, for example, \cite{CK,MAN5}). Note that the scaling vector  $\alpha$ has the most influence on the  fractal dimension of the graph of $f_{\Delta,b}^\alpha$ and hence on the ``roughness" of the function $f_{\Delta,b}^\alpha$. For instance, we have the following proposition given in \cite{Akhtar}.
\begin{theorem}\textup{\cite[Corollary 3.1]{Akhtar}}.\label{Akhtar}
Let $f$ and $b$  be Lipschitz continuous functions defined on $I$ with $b(x_0)=f(x_0)$ and $b(x_N)=f(x_N).$ Let $\Delta=\{x_0,x_1,\dots , x_N\}$ be a partition of $I=[x_0,x_N]$ satisfying $x_0<x_1< \dots < x_N$ and $ \alpha =(\alpha_1,\alpha_2, \dots ,\alpha_{N}) \in (-1,1)^N.$ If the data points $ \{(x_i, f(x_i)): i =0,1 \dots, N\}$ are not collinear, then the graph of the $\alpha$-fractal function $f_{\Delta,b}^{\alpha}$ denoted by
$G_{f_{\Delta,b}^\alpha}$ has the box dimension
\begin{equation*}
\dim_B(G_{f_{\Delta,b}^\alpha})=
                   \begin{cases} D,  \text{ if $\sum_{i=1}^{N} |\alpha_i| > 1$}\\
                      1, \text{ otherwise ,}
                  \end{cases}
                    \end{equation*}
where $D$ is the solution of $\sum_{i=1}^{N} |\alpha_i|a_i^{D-1}=1.$
\end{theorem}
Therefore, it appears that the roughness in the constructed fractal interpolant (approximant) and the convergence (closeness) to the original function  may not be simultaneously achieved.
\par
One can circumvent this by exploring the choices of other parameters in the construction of the fractal function $f_{\Delta,b}^\alpha$. For instance, a look back at the estimate in (\ref{eq1}) should convince the reader that with any permissible choice of $\alpha$,   the fractal function $f_{\Delta,b}^\alpha$ is close to the seed function $f$, provided the base function $b$ is close enough to $f$. In particular, if $(b_n)_{n \in \mathbb{N}}$ is a sequence of base functions satisfying $b_n(x_0)=f(x_0)$ and $b_n(x_N)=f(x_N)$  for all $n\in \mathbb{N}$ and $b_n \to f$ as $n \to \infty$, then $f_{\Delta,b_n}^\alpha \to f$ uniformly as $n \to \infty.$ For instance, one can take $b_n= B_n(f)$, where $B_n: \mathcal{C}(I) \to \mathcal{C}(I)$ defined by
$$B_nf (x) = \sum_{k=0}^n f\Big(x_0+ \frac{k}{n}(x_N-x_0)\Big) {n \choose k} \frac{(x-x_0)^k (x_N-x)^{n-k}}{(x_N-x_0)^n},$$
is the classical Bernstein operator. This simple but noteworthy observation was exploited  to define what is called \emph{Bernstein $\alpha$-fractal functions} corresponding to $f$, denoted by $f_{\Delta, B_n(f)}^\alpha =f_n^\alpha$ \cite{Vij1}. Theorem $2$ in reference \cite{Vij1}, which reads as follows, contains an error in the statement.
\begin{theorem}\textup{\cite[Theorem 2]{Vij1}}.
Let $\Delta=\{x_0,x_1,\dots , x_N\}$ be a partition of $I=[x_0,x_N]$ satisfying $x_0<x_1< \dots < x_N$ and $ \alpha =(\alpha_1,\alpha_2, \dots ,\alpha_{N}) \in (-1,1)^N.$ For an irregular function $f \in \mathcal{C}(I)$ if all the Bernstein $\alpha$-fractal functions in the sequence $(f_n^\alpha)_{n \in \mathbb{N}}$ are obtained with a fixed choice of scaling vector $\alpha$ whose components satisfy $\sum_{i=1}^N |\alpha_i| >1$, then all the Bernstein $\alpha$-fractal functions in the sequence $(f_n^\alpha)_{n \in \mathbb{N}}$  have the same fractal dimension and $f_n^\alpha \to f$ uniformly as $n \to \infty.$
\end{theorem}
 The author claims that the above theorem follows from Theorem \ref{Akhtar}, but overlooked that Theorem \ref{Akhtar} needs additional assumptions, for instance, the seed function $f$ has to be Lipschitz. We shall refine this theorem by inserting the required additional condition.  While the original proof holds with this additional assumption, we also provide a revised proof that shows that the assumption is not needed. We take the opportunity to observe that similar adjustments must be made to the identical results appeared elsewhere;  see, for instance,  \textup{\cite[Theorem 2.3]{Vij2}} and \textup{\cite[Theorem 2]{Vij3}}.
 \par
Let the partition $\Delta$ and scale vector $\alpha$ be fixed. If a suitable sequence $(b_n)_{n \in \mathbb{N}}$ is used in the place of a single base function $b$, then corresponding to a fixed $f \in \mathcal{C}(I)$  we obtain a family of fractal functions $\{f_{\Delta,b_n}^\alpha:n \in \mathbb{N}\}$. In this case, corresponding to the fractal operator $\mathcal{F}_{\Delta,b}^\alpha=\mathcal{F}^\alpha$ in Equation (\ref{eq0}) one obtains a multi-valued operator or  a set-valued operator
\begin{equation}\label{eq3}
\mathcal{F}^\alpha: \mathcal{C}(I) \rightrightarrows \mathcal{C}(I), \quad \mathcal{F}^\alpha(f)= \{f^\alpha_{\Delta,b_n}: n\in \mathbb{N}\}.
\end{equation}
In  \textup{\cite[Theorem 3]{Vij3}} the author  attempts to prove that the aforementioned set-valued operator is linear and bounded. It seems that in the proof, the operator is treated as a single-valued operator. This is the case, for instance, for each fixed $n \in \mathbb{N}$. However, to the best of our knowledge the linearity and boundedness of a multi-valued operator need to be approached in a different way. A \emph{closed convex process} is treated as a set-valued analogue of a continuous (bounded) linear operator in the sense that a closed convex processes enjoy almost all properties of  continuous linear operators, including the open mapping theorem, closed graph theorem and uniform boundedness principle; see, for instance, \cite{Aubin}. We prove that the multi-valued fractal operator in (\ref{eq3}) is a \emph{process} and \emph{Lipschitz} but not \emph{linear}, where linearity is interpreted in an appropriate sense. Similarly, in \textup{\cite[Theorem 4]{Vij3}}, with a suitable assumption on the scaling vector, the author attempts to prove that the multi-valued operator $\mathcal{F}^\alpha$ in (\ref{eq3}) is bounded below, but not compact. But, in the ``proof" of this theorem, $\mathcal{F}^\alpha$ is treated as a single-valued operator. A possible explanation of this could be that the author deals, or rather intends to deal, with the single-valued operator $\mathcal{F}_n^\alpha: \mathcal{C}(I) \to \mathcal{C}(I)$ defined by $\mathcal{F}_n^\alpha(f)=f^\alpha_{\Delta,b_n}$ for each fixed $n \in \mathbb{N}$. Another observation worth noting regarding \textup{\cite[Theorem 4]{Vij3}} is that in case one intends to handle the single-valued operator $\mathcal{F}_n^\alpha: \mathcal{C}(I) \to \mathcal{C}(I)$,  thanks to items (2) and (3) of Theorem \ref{prelthm} above, the assumptions on the scaling vector made in  \textup{\cite[Theorem 4]{Vij3}} can be dropped. We collect all these refinements that act as a corrigendum to \cite{Vij1,Vij2,Vij3} in the last section.

\section{Fractal Rational Trigonometric Functions}

 We consider here the space of $2\pi$-periodic continuous functions $$ \mathcal{C}(2\pi)= \big\{f:[-\pi,\pi] \rightarrow \mathbb{R}; f ~ \text{is ~continuous}, f(-\pi)=f(\pi)\big\}.$$ Let $\mathfrak{T}_m(2 \pi)$ be the set of trigonometric polynomials of degree at most $m.$ Recall that $\mathfrak{T}_m(2 \pi)$ is linearly spanned by the set $$\big\{1, \sin x,\cos x,\sin 2x,\cos 2x,\dots , \sin mx,\cos mx\big\}.$$ In fact, this family constitutes a basis for $\mathfrak{T}_m(2 \pi)$ and this system is orthogonal with respect to the standard inner product
$$ \langle f, g \rangle := \int_{- \pi }^\pi f(x) g(x) \mathrm{d}x.$$
Let $\Delta: -\pi=x_0<x_1< \dots < x_N=\pi $ be a partition of the interval $I=[-\pi,\pi].$
The following class of functions is introduced in \cite{MAN3}.
\begin{definition}\textup{\cite[Definition 4.1]{MAN3}}.\label{basicdef1}
Let $m$ be a nonnegative integer. We define the set of $\alpha-$fractal trigonometric polynomials of degree at most $m$ denoted by $\mathfrak{T}_m^{\alpha}(2 \pi)$ as $\mathcal{F}_{\Delta,L}^{\alpha}\big(\mathfrak{T}_m(2 \pi)\big)$, where $\mathcal{F}_{\Delta,L}^\alpha= \mathcal{F}^\alpha$ is the (single-valued) $\alpha$-fractal operator defined in (\ref{eq0}).  An element in $\mathfrak{T}_m^\alpha(2 \pi)$ is referred to as an $\alpha$-fractal trigonometric polynomial or simply as a fractal trigonometric polynomial. Further, the set of all $\alpha$-fractal trigonometric polynomials is defined as $\mathfrak{T}^\alpha (2 \pi)= \cup_{m}\mathfrak{T}_m^{\alpha}(2 \pi)$.
\end{definition}
Similar to the class of trigonometric polynomials, one can define rational trigonometric functions in $\mathcal{C}(2 \pi)$ as follows.
For $m, n \in \mathbb{N}\cup \{0\}$, let $$\mathfrak{R}_{mn}(2 \pi):= \Big\{t=\frac{p}{q}: p \in \mathfrak{T}_m, q \in \mathfrak{T}_n~\text{and}~ q>0 ~ \text{on} ~ [-\pi, \pi] \Big\},$$  the set of all real-valued rational trigonometric functions of type $(m,n)$ and $$ \mathfrak{R}(2\pi)= \cup_{m,n} \mathfrak{R}_{mn}(2\pi).$$
Following the construction of fractal versions of classical functions such as polynomials, trigonometric functions and rational functions \cite{MAN2,MAN3,Rational}, in the upcoming definition we apply the fractal operator to map the class of rational trigonometric functions to its fractal counterpart.
\begin{definition}
For $m, n \in \mathbb{N}\cup \{0\}$ we define the class of \emph{$\alpha-$fractal rational trigonometric functions} of type $(m,n)$ denoted by $\mathfrak{R}^{\alpha}_{mn}(2\pi)$ as the image of $\mathfrak{R}_{mn}(2 \pi)$ under the fractal operator $\mathcal{F}_{\Delta,L}^{\alpha}$. That is,
$$\mathfrak{R}_{mn}^\alpha(2\pi):= \mathcal{F}_{\Delta,L}^{\alpha}\big(\mathfrak{R}_{mn}(2 \pi)\big).$$
Further, we let
$$\mathfrak{R}^{\alpha}(2\pi)= \mathcal{F}_{\Delta,L}^{\alpha}\big(\mathfrak{R}(2\pi)\big),$$ the set of all $\alpha-$fractal rational trigonometric functions.
\end{definition}
 \begin{remark}
 We can also define a new class of $\alpha-$fractal rational trigonometric functions in the following way $$\mathfrak{S}^{\alpha}_{mn}(2\pi)= \Big\{\frac{p^{\alpha}}{q^{\alpha}}: p^{\alpha} \in \mathfrak{T}_m^{\alpha}, q^{\alpha} \in \mathfrak{T}_n^{\alpha}~\text{and}~ q^{\alpha}>0 ~ \text{on} ~ I \Big\}.$$
 For suitable choices of the scale vector, one can obtain $q^{\alpha}> 0$ on $I$ whenever so is $q$ \cite{PV2}. A difference in the two classes of fractal functions $\mathfrak{R}^{\alpha}_{mn}(2\pi)$ and $\mathfrak{S}^{\alpha}_{mn}(2\pi)$ defined above is the following. It is evident that a function $r^\alpha$ in $\mathfrak{R}^{\alpha}_{mn}(2\pi)$ satisfies the self-referential equation of the form
  $$r_{\Delta,L}^\alpha(x) = r(x) + \alpha_i (r_{\Delta,L}^{\alpha}- Lr)\big(L_i^{-1}(x)\big) ~ ~ ~~\forall~~ x \in I_i,~~ i \in \mathbb{N}_N,$$
and its graph is the attractor of an IFS whereas it is not certain if a similar self-referentiality is applicable for a function in $\mathfrak{S}^{\alpha}_{mn}(2\pi)$.  Please consult also Section \ref{nonselfssec} of this article.
\end{remark}
\begin{remark}
 To obtain a more general class of FIFs, the constant scaling factors $\alpha_i \in (-1,1)$ can be replaced by functions $\alpha_i  \in \mathcal{C}(I)$ such that
 $ \|\alpha_i\|_\infty< 1$  for all $i \in \mathbb{N}_N$ \cite{WY}. Correspondingly, for a given $f \in \mathcal{C}(I)$, with $\alpha:= (\alpha_1, \alpha_2, \dots, \alpha_N)$, we can define an $\alpha$-fractal function $f_{\Delta,b}^\alpha=f^\alpha$ satisfying the self-referential equation
 $$f^\alpha(x) = f(x) + \alpha_i(L_i^{-1}(x)) (f^{\alpha}- b)\big(L_i^{-1}(x)\big) ~ ~ ~~\forall~~ x \in I_i,~~ i \in \mathbb{N}_N.$$
 In the sequel, we shall use constant scaling factors but we remark here that most of our results can be applied to the setting of variable scaling factors as well.
 \end{remark}

\begin{example} Let $p(x)=27 \sum_{k=0}^{2}\sin(x_{k3})J_3^2(x -x_{k3})$ and $q(x)=19 + 8 \cos(3x)$ for  $x \in [0,1] $. Here $x_{k3}=\frac{2k\pi}{3} $ for $k=0,1,2,$ and $J_3$ is the Jackson function (see, Section \ref{AEB}). The rational trigonometric function $r(x)=\dfrac{p(x)}{q(x)}$ is plotted in Fig. \ref{figFTR}(a). We consider the partition $\Delta: 0 < \frac{1}{10}<\frac{2}{10}< \dots < \frac{9}{10}<1$ and scale vector $\alpha$ with components $\alpha_i=0.9$ for $i \in \mathbb{N}_{10}$. Figs.  \ref{figFTR}(b)-(c) correspond to the fractal rational trigonometric function $r_{\Delta,L}^\alpha$ with $L$ defined by
\begin{enumerate}[(i)]

\item $Lf= \nu f$, where $\nu(x)= 1+x(x-1)$
\item $Lf= f\circ \varphi$, where $\varphi(x)= x^3$.

\end{enumerate}
Fig. \ref{figFTR}(d) depicts two graphs one (red color) corresponds to the self-referential rational trigonometric function $r_{\Delta,L}^\alpha$ with parameters as in Fig.\ref{figFTR}(b) and the other (blue color)  corresponds to $\dfrac{p_{\Delta,L}^{\alpha}(x)}{q_{\Delta,L}^{\alpha}(x)} $, where the fractal functions $p^\alpha$ and $q^\alpha$ are constructed with same $\Delta$ and $\alpha$ as before, and $L$ as in item (1) above.

\end{example}

 \begin{figure}[h!]
 \begin{center}
 \begin{minipage}{0.47\textwidth}
 \epsfig{file=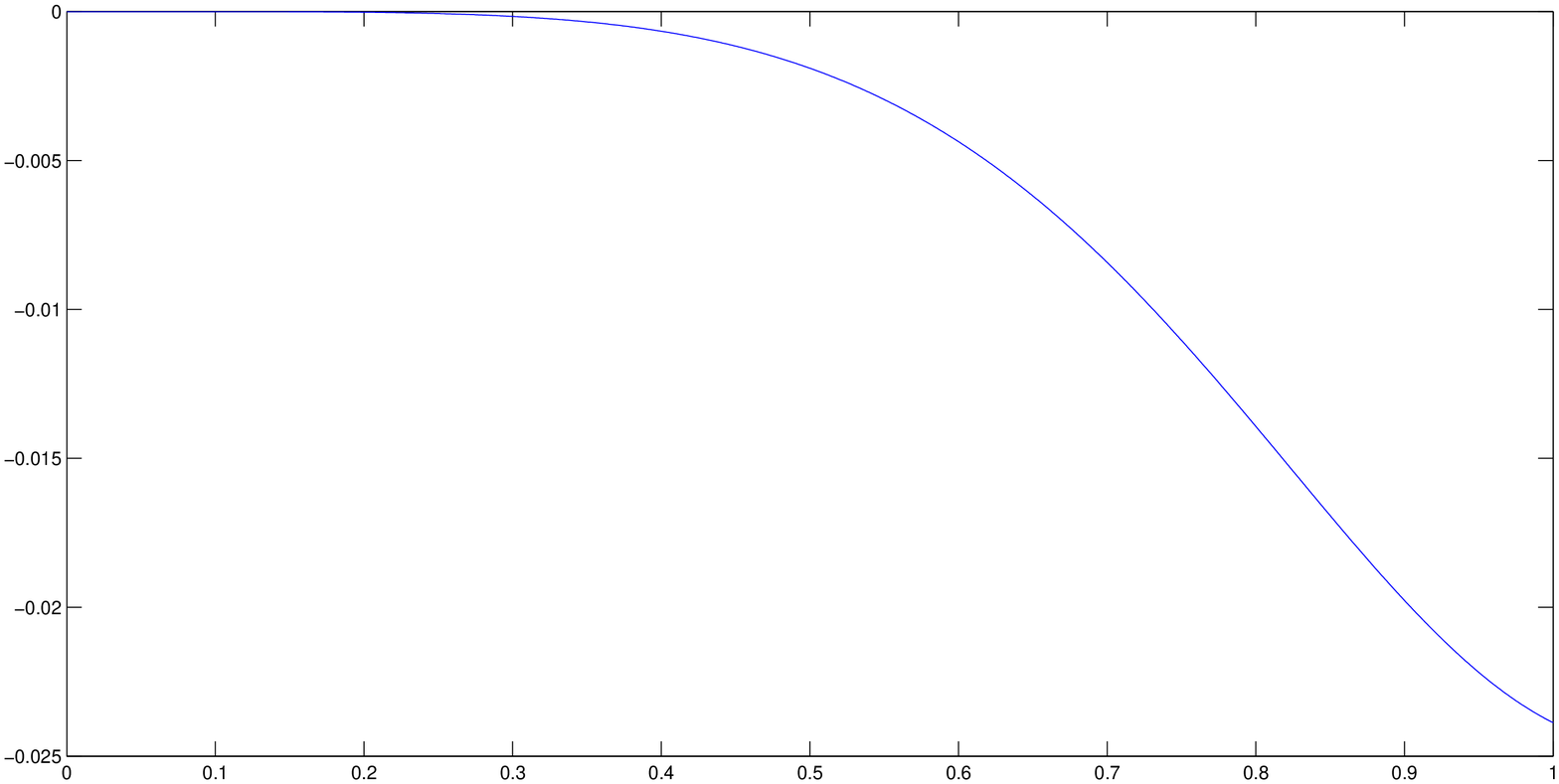,scale=0.17}
  \centering{(a)}
 \end{minipage}\hfill
 \begin{minipage}{0.47\textwidth}
 \epsfig{file=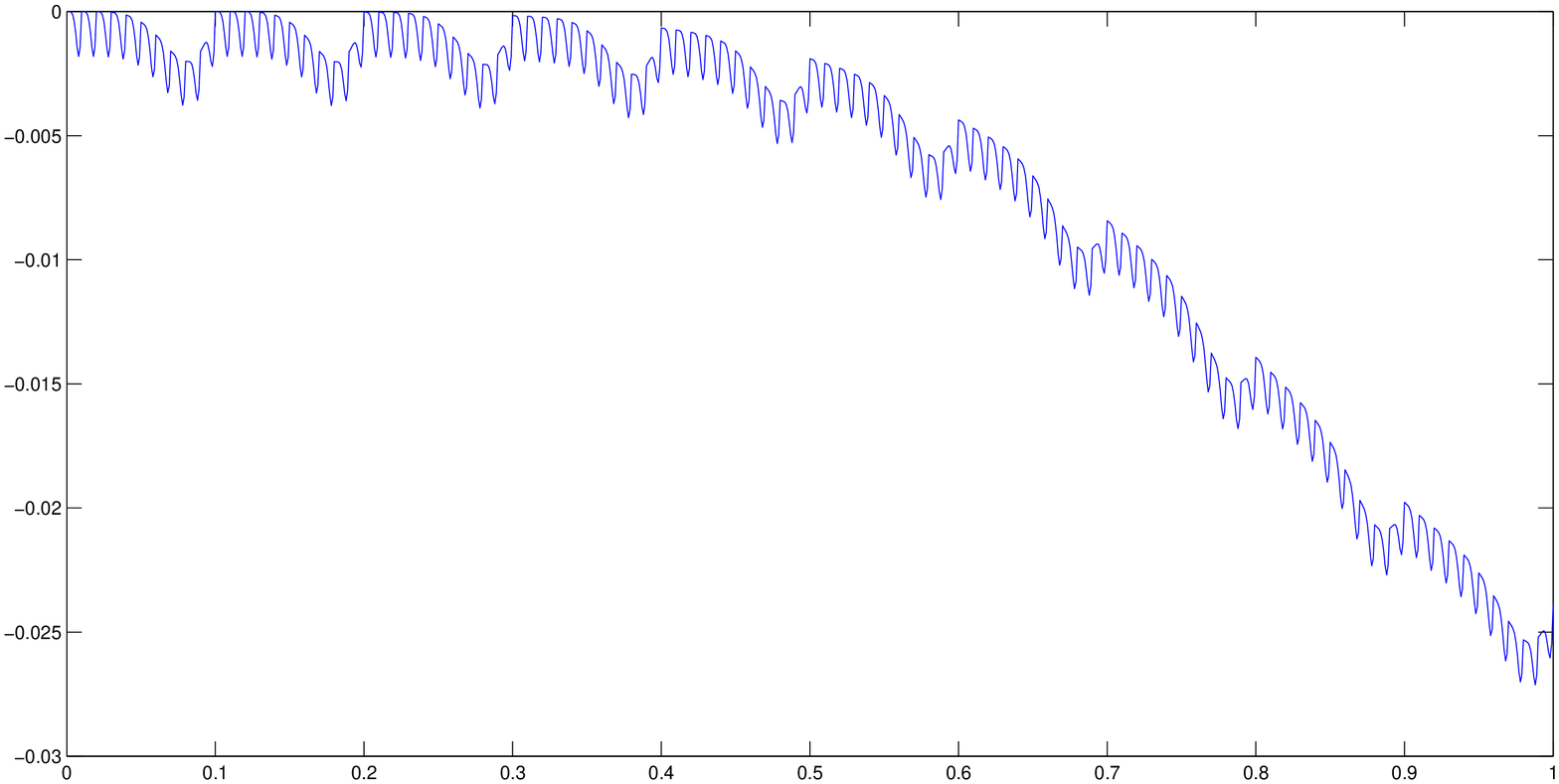,scale=0.17}
 \centering{(b)}
 \end{minipage}\\
 \begin{minipage}{0.47\textwidth}
 \epsfig{file = 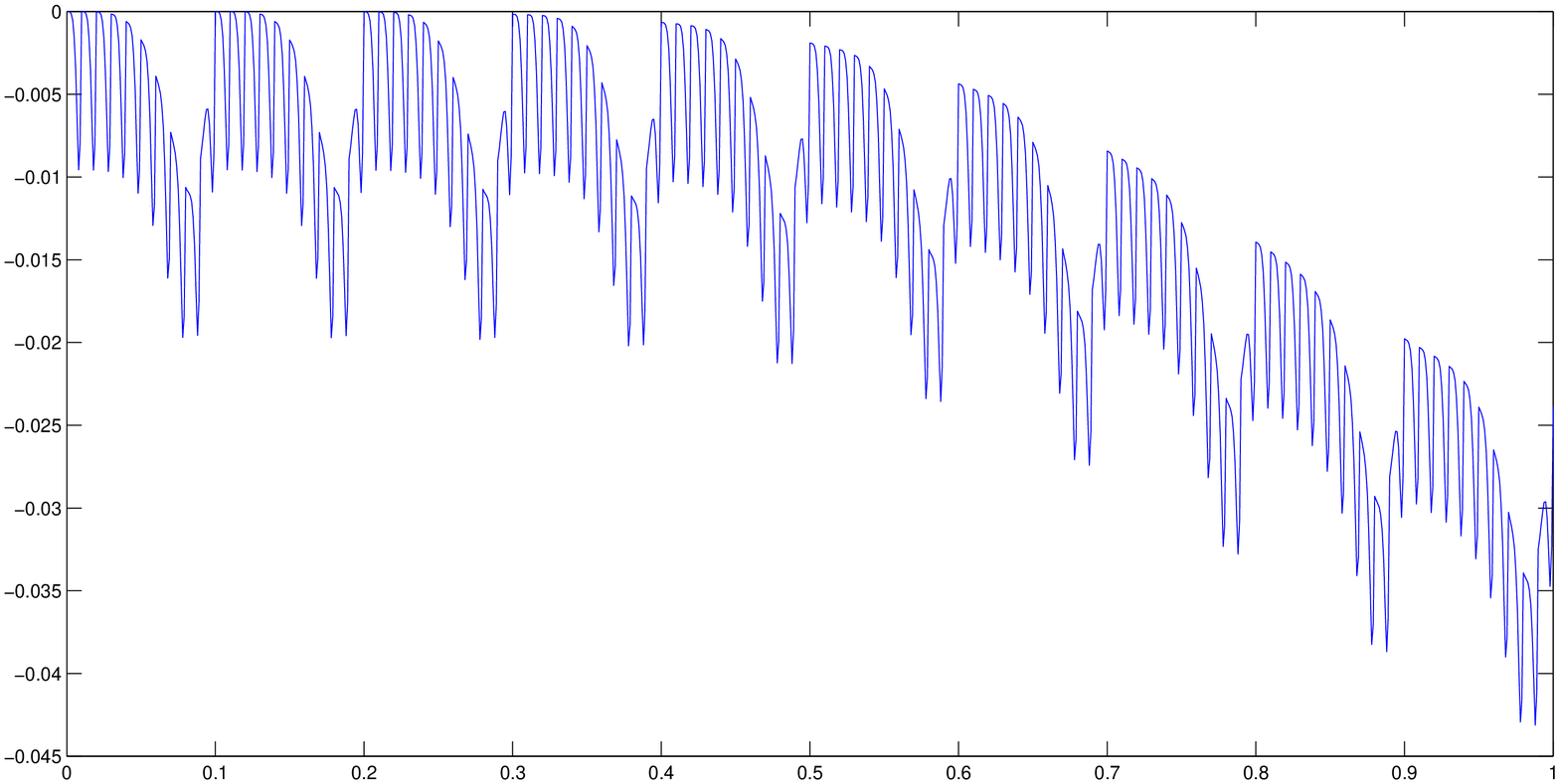,scale=0.17}
 \centering{(c)}
 \end{minipage}\hfill
 \begin{minipage}{0.47\textwidth}
 \epsfig{file=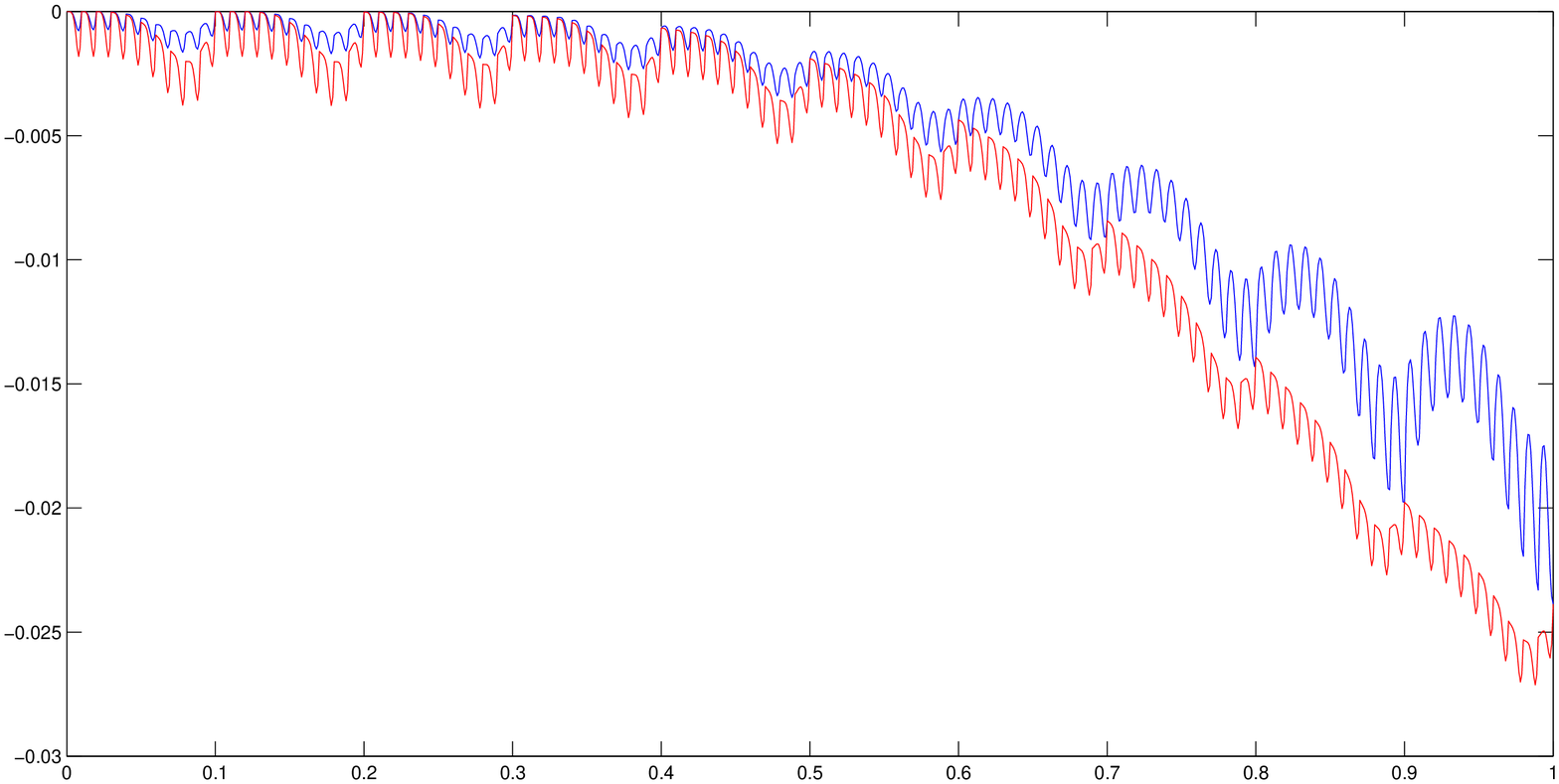,scale=0.17}
 \centering{(d)}
 \end{minipage}
 \caption{A rational trigonometric function and its fractal versions.}\label{figFTR}
 \end{center}
 \end{figure}

\subsection{Weierstrass-type theorems}
In the following theorems we show that a given $f \in \mathcal{C}(2 \pi)$  can be uniformly well-approximated by a fractal rational trigonometric function. The idea is to apply Inequality (\ref{eq2}) to find suitable parameters that provide a close enough fractal perturbation $t_{\Delta,L}^\alpha$ of  a rational trigonometric function $t$ that well approximates $f$. This basic idea is not claimed to be new, and is, in fact, explored in various contexts scattered in the fractal approximation literature (see, for instance, \cite{MAN1,Rational,Vij1}).

\begin{theorem}\label{WLT1}
 Let $f \in \mathcal{C}(2\pi)$ and $\epsilon>0.$ Suppose that the  partition $\Delta:= \{x_0,x_1, \dots,x_N: x_0< x_1< \dots <x_N\}$ of the interval $I=[-\pi,\pi]$, and the bounded linear operator $L:\mathcal{C}(2\pi)\rightarrow \mathcal{C}(2\pi),$ $L \ne Id$ satisfying $(Lf)(x_0)=f(x_0),(Lf)(x_N)=f(x_N)$ are arbitrary, but fixed.  Then there exists a scale vector $\alpha=\alpha(\epsilon)$ in $(-1,1)^N$, $\alpha \neq 0$
 and an $\alpha$-fractal rational trigonometric function $t_{\Delta,L}^{\alpha}$ such that
$$ \|f- t_{\Delta,L}^{\alpha}\|_{\infty} <\epsilon .$$
\end{theorem}
\begin{proof}
Let $\epsilon >0$ be given. By the Stone-Weierstrass theorem, there exists a rational trigonometric function $ t \in \mathcal{R}(2\pi)$ such that $$ \|f- t\|_{\infty} <\frac{\epsilon}{2} .$$ For a partition $ \Delta:= \{x_0,x_1, \dots,x_N: x_0< x_1< \dots <x_N\} $ of $I=[x_0,x_N]=[-\pi,\pi]$ and for a bounded linear operator $L:\mathcal{C}(2\pi)\rightarrow \mathcal{C}(2\pi),$ $L \ne I_d$ satisfying $(Lf)(x_0)=f(x_0),(Lf)(x_N)=f(x_N),$ select $ \alpha \in (-1,1)^{N}$, $\alpha \neq 0$ such that $$ |\alpha |_{\infty} <\frac{\frac{\epsilon}{2}}{\frac{\epsilon}{2}+\|Id-L\| \|t\|_{\infty}} .$$
Then we have
\begin{equation*}
  \begin{aligned}
   \|f- t_{\Delta,L}^{\alpha}\|_{\infty} & \leq \|f- t\|_{\infty}+\|t- t_{\Delta,L}^{\alpha}\|_{\infty}\\
   & \leq \|f- t\|_{\infty}+ \frac{|\alpha |_{\infty}}{1-|\alpha |_{\infty}}\|Id- L\| \|t\|_{\infty}\\
   & < \frac{\epsilon}{2} + \frac{\epsilon}{2}\\
    & = \epsilon,
  \end{aligned}
  \end{equation*}
  completing the proof.
\end{proof}
\begin{theorem}\label{WLT2}
 Let $f \in \mathcal{C}(2\pi)$ and $\epsilon>0. $ Let the partition $\Delta:= \{x_0,x_1, \dots,x_N: x_0< x_1< \dots <x_N\}$ of the interval of $I=[x_0,x_N]=[-\pi,\pi]$ and scale vector  $ \alpha \in (-1,1)^N$ be arbitrary but fixed. Then, there exists a bounded linear operator $L: \mathcal{C}(2\pi)\rightarrow \mathcal{C}(2\pi),$ $L \ne I_d$ satisfying $(Lf)(x_0)=f(x_0),(Lf)(x_N)=f(x_N)$ and an
 $\alpha$-fractal rational trigonometric function $t_{\Delta,L}^{\alpha}$ such that
$$ \|f- t_{\Delta,L}^{\alpha}\|_{\infty} < \epsilon .$$
\end{theorem}
\begin{proof}
Let $\epsilon >0$ be given. By the Stone-Weierstrass theorem, there exists a rational trigonometric function $ t \in \mathcal{C}(2\pi)$ such that $$ \|f- t\|_{\infty} <\frac{\epsilon}{2} .$$ Choose a partition $ \Delta:= \{x_0,x_1, \dots,x_N: x_0< x_1< \dots <x_N\} $ of $I=[x_0,x_N]=[-\pi,\pi]$ and a scale vector  $ 0 \ne \alpha \in (-1,1)^{N}$ satisfying $ |\alpha|_{\infty} <1.$ Now let us consider a bounded linear operator $L:\mathcal{C}(2\pi)\rightarrow \mathcal{C}(2\pi),$ $L \ne Id$ satisfying $(Lf)(x_0)=f(x_0),(Lf)(x_N)=f(x_N),$ such that $$ \|Id-L\|  <\frac{1-|\alpha |_{\infty}}{ |\alpha |_{\infty} \|t\|_{\infty}} \frac{\epsilon}{2} .$$
Then we have
\begin{equation*}
  \begin{aligned}
   \|f- t_{\Delta,L}^{\alpha}\|_{\infty} & \leq \|f- t\|_{\infty}+\|t- t_{\Delta,L}^{\alpha}\|_{\infty}\\
   & \leq \|f- t\|_{\infty}+ \frac{|\alpha |_{\infty}}{1-|\alpha|_{\infty}}\|Id- L\| \|t\|_{\infty}\\
   & < \frac{\epsilon}{2} + \frac{\epsilon}{2}\\
   &=\epsilon,
    \end{aligned}
     \end{equation*}
    and this completes the proof.
    \end{proof}
\begin{remark}
Let $f \in \mathcal{C}(2\pi)$. The above theorems, in particular, assert the following.
  \begin{enumerate}
  \item Let $\alpha^m \in \mathbb{R}^N$, $|\alpha^m|_\infty<1$ and $\alpha^m \to 0$ as $m \to \infty$. Then there exists a sequence of fractal rational trigonometric functions $\big(t_{\triangle, L}^{\alpha^m}\big)$ which converges to $f$ uniformly.
  \item Let $(L_n)_{n \in \mathbb{N}}$ be a sequence of bounded linear operators on $\mathcal{C}(2\pi)$ satisfying $ L_n(g) \to g $ for each $g \in \mathcal{C}(2\pi)$. Then
there exists a sequence of fractal rational trigonometric functions $\big(t_{\triangle, L_n}^{\alpha}\big)$ which converges to $f$ uniformly. For instance, one can work with Bernstein operators corresponding to $f$.
\end{enumerate}
\end{remark}
\begin{theorem} \label{newthm1}
Let $\mathfrak{R}_{\Delta, B_n}^{\alpha}(2 \pi)=\mathcal{F}_{\Delta, B_n}^{\alpha} \big( \mathfrak{R} (2 \pi) \big)$ be the class of all
 $\alpha$-fractal rational trigonometric functions with a fixed choice of the scale vector $\alpha$, partition $\Delta$ and Bernstein operator $B_n$.
The set $ \bigcup_{n \in \mathbb{N}} \mathfrak{R}_{\Delta, B_n}^{\alpha}(2 \pi)$ is dense in $\mathcal{C}(2\pi).$
\end{theorem}
\begin{proof}
The proof is immediate from item (2) of the previous remark.
\end{proof}
The following theorem demonstrates that a continuous non-negative function on a compact interval can be uniformly well-approximated by a non-negative $\alpha$-fractal rational trigonometric function. A similar result in the setting of $\alpha$-fractal rational function can be consulted in \textup{\cite[Theorem 3.5]{Rational}}. Although the proof is patterned after \textup{\cite[Theorem 3.5]{Rational}},  the difference lies in the fact that in the following theorem, the scale vector  $\alpha$ is arbitrary, except that $|\alpha|_\infty<1$.
\begin{theorem}
Let $f \in \mathcal{C}(2\pi)$ be such that $f(x) \geq 0$ for all $x \in I=[-\pi,\pi].$ Then for any $\epsilon >0 ,$ and for any $\alpha \in (-1,1)^N$, there exists a nonnegative $\alpha$-fractal rational trigonometric function $ t_{\Delta,L}^{\alpha}$ such that $ \|f- t_{\Delta,L}^{\alpha}\|_{\infty} <\epsilon .$ A similar result holds for a continuous non-positive function.
\end{theorem}
\begin{proof}
Let $\epsilon >0 $ and $ f \in \mathcal{C}(2\pi)$ be such that $f(x) \geq 0$ for all $x \in I.$  We assume further that the operator $L$ used in the construction of $f_{\Delta,L}^\alpha$ fixes the constant function $1$ defined by $ 1(x)=1 $ for all $x \in I.$ That is, $L(1)=1.$ For instance, note that the Bernstein operators $B_n$ fixes the function $f(x)=1$. Assume $ \alpha \in (-1,1)^{N}.$ From the self-referential equation for $f_{\Delta,L}^\alpha$, we obtain
$$ \|f_{\Delta,L}^{\alpha}- f\|_{\infty} \leq |\alpha |_{\infty}\|f_{\Delta,L}^{\alpha}- Lf\|_{\infty}.$$ For $f=1$, the above inequality gives $ \|f_{\Delta,L}^{\alpha}- 1\|_{\infty} \leq |\alpha |_{\infty}\|f_{\Delta,L}^{\alpha}- 1\|_{\infty}$ and this gives $\|f_{\Delta,L}^{\alpha}- 1\|_{\infty} = 0 .$ Therefore, $f_{\Delta,L}^\alpha=1$, that is $\mathcal{F}_{\Delta,L}^{\alpha}(1)=1.$\\
For $\epsilon >0 $, $\alpha \in (-1,1)^N$ and $ f \in \mathcal{C}(2\pi).$ In view of Theorem \ref{WLT2}, there exists a rational trigonometric function $t_{\Delta,L}^\alpha$ such that $$ \|f- t_{\Delta,L}^{\alpha}\|_{\infty} <\frac{\epsilon}{2}, ~ \text{ where}~ \mathcal{F}_{\Delta,L}^{\alpha}(t)=t_{\Delta,L}^{\alpha}.$$ Define
 $ r_{\Delta,L}^{\alpha}(x)=t_{\Delta,L}^{\alpha}(x)+ \frac{\epsilon}{2} $ for all $x \in I .$ Since $1$ is a fixed point of $\mathcal{F}_{\Delta,L}^{\alpha},$   $$ r_{\Delta,L}^{\alpha}(x)=t_{\Delta,L}^{\alpha}(x)+ \frac{\epsilon}{2}1(x) = t_{\Delta,L}^{\alpha}(x)+ \frac{\epsilon}{2}1^{\alpha}(x).$$
 Further, since $\mathcal{F}_{\Delta,L}^{\alpha}$ is a linear operator $$ r_{\Delta,L}^{\alpha} = t_{\Delta,L}^{\alpha}+ \frac{\epsilon}{2}1^{\alpha}= \mathcal{F}_{\Delta,L}^{\alpha}(t+\frac{\epsilon}{2} 1).$$
 The above equation tells that $ r_{\Delta,L}^{\alpha}$ is a fractal rational trigonometric polynomial. Further, we have $$ r_{\Delta,L}^{\alpha}(x)=t_{\Delta,L}^{\alpha}(x)+ \frac{\epsilon}{2} = t_{\Delta,L}^{\alpha}(x)+ \frac{\epsilon}{2}-f(x)+f(x) \geq f(x)+ \frac{\epsilon}{2} - \| t_{\Delta,L}^{\alpha}- f \|_{\infty} \geq 0. $$
 Moreover, we obtain
 $$\| f-r_{\Delta,L}^{\alpha} \|_{\infty}\leq \| f-t_{\Delta,L}^{\alpha} \|_{\infty}+\| t_{\Delta,L}^{\alpha}-r_{\Delta,L}^{\alpha} \|_{\infty}< \frac{\epsilon}{2}+\frac{\epsilon}{2} = \epsilon $$ and hence the proof.
 \end{proof}
\begin{remark}
An analogous result can be proved for $\alpha$-fractal rational function, which can be treated as an improvement to  \textup{\cite[Theorem 3.5]{Rational}}
in the sense that the scale vector $\alpha$ is arbitrary.
\end{remark}
\section{Best Approximation Property of $\mathfrak{R}^{\alpha}_{mn}(2\pi)$}
\begin{definition} \textup{\cite[p. 372]{DVpai}}.
Let $(X,\|.\|)$ be a normed linear space over $\mathbb{K},$ the field of real or complex numbers. Given a nonempty set $V \subset X$ and an element $x \in X ,$ distance from $V$ to $x$ is defined as $$d(x,V)=\inf\{\|x-y\| : y\in V\}.$$ An element $v(x) \in V$ such that $\|x- v(x)\|= d(x,V)$ if it exists  is called a \emph{best approximant} to $x$ from $V.$ A subset $V$ of $X$ is called \emph{proximinal (proximal or existence set)} if for each $x \in X$ a best approximant $v(x) \in V$ of $x$ exists.
\end{definition}

We recall a well-known fact (see, for example, \cite{DVpai,Cheney}) that
\begin{theorem}\textup{\cite[p. 20]{Cheney}}.\label{BATthm1}
Let $X$ be a normed linear space and $E$ be  a finite dimensional  subspace of $X$. Then $E$ is  proximinal in $X$, that is,  for each $x$ in $X$, a best approximant from $E$ to $x$ exists.
\end{theorem}
\begin{remark}
Since $\mathfrak{T}_m$ and consequently
$\mathfrak{T}_m^{\alpha}$ is a finite dimensional subspace of $\mathcal{C}(2\pi)$, it follows that for each $f \in \mathcal{C}(2\pi)$, the best approximant $t^\alpha_m (f)$ from $\mathfrak{T}_m^{\alpha}$ to $f$ exists. That is,
$$ \| f- t^\alpha_m (f)\|_\infty = \inf\big\{\|f - t^{\alpha}_m \|_{\infty}:t^{\alpha}_m \in \mathcal{T}^{\alpha}_m \big\}.$$
\end{remark}
\begin{definition}\textup{\cite[p. 71]{DVpai}}.
If $V$ is a proximinal subset of $X$ and the best approximant $v^*(x)$ for each $x \in X$ is unique, then we can define a map $P_V:X \rightarrow V $ by $ P_V(x)=v^*(x).$ This map is called as \emph{a best approximation operator}.
In general, best approximant to $x$ from $V$ is not unique, therefore, $ P_V(x)$ is the set of all best approximants to $x$ from $V.$ The set valued map $ P_V:X \rightrightarrows V$ is called the \emph{metric projection} supported on $V.$
\end{definition}
In this section, we shall establish that $\mathfrak{R}^{\alpha}_{mn}(2\pi)$ is a proximinal subset of $\mathcal{C}(2\pi),$ i.e., for each $f \in \mathcal{C}(2\pi),$ there exists an element $r_*^{\alpha} \in \mathfrak{R}^{\alpha}_{mn}(2\pi)$ such that $$\|f- r_*^{\alpha}\|_{\infty} =\text{dist} \big(f, \mathfrak{R}^{\alpha}_{mn}(2\pi)\big) :=\inf\big\{\|f-r^{\alpha}\|_{\infty} : r^{\alpha} \in \mathfrak{R}^{\alpha}_{mn}(2\pi) \big\}.$$ Note that $\mathfrak{R}^{\alpha}_{mn}(2\pi)$ is not a linear subspace of $\mathcal{C}(2\pi)$ and hence in contrast to the case $\mathfrak{T}_m^{\alpha} (2 \pi)$, Theorem \ref{BATthm1} cannot be applied to infer that $\mathfrak{R}^{\alpha}_{mn}(2\pi)$ is proximinal. First let us record the following definition and lemma.
\begin{definition}\textup{\cite[p. 376]{DVpai}}.
Let $Y$ be a non-empty subset of a normed linear space $X.$ Then $Y$ ia said to be approximately compact if for every $x \in X,$ each sequence $ (y_n) \subseteq Y$ such that $\| x-y_n\| \rightarrow d(x,Y),$ has a subsequence convergent in $Y.$
\end{definition}
\begin{lemma}\textup{\cite[p. 156]{Cheney}}. Let $P$ and $Q$ be two non-zero trigonometric polynomials with real coefficients such that $|P(\theta)|\leq |Q(\theta)|$ for all real $\theta$. If $Q$ has a real zero, then there exist non-zero trigonometric polynomials $P^*$ and $Q^*$ with real coefficients such that $\deg(P^*) < \deg(P) , \deg(Q^*) < \deg(Q)$ and $P^*Q=PQ^*.$
\end{lemma}
\begin{theorem}\label{proximal}
If the fractal operator $\mathcal{F}_{\Delta,L}^{\alpha}$ is bounded below, then for each $f \in \mathcal{C}(2\pi)$ there exists a fractal rational trigonometric function $r_*^{\alpha} \in \mathfrak{R}^{\alpha}_{mn}(2\pi)$ such that $\|f- r_*^{\alpha}\|_{\infty} = \text{dist} \big(f, \mathfrak{R}^{\alpha}_{mn}(2\pi) \big)$. In particular, $\mathfrak{R}^{\alpha}_{mn}(2\pi) $ is approximately compact.
\end{theorem}
\begin{proof}
Let $ d = \text{dist} (f, \mathfrak{R}^{\alpha}_{mn}(2\pi) ).$ By the definition of infimum, we get a sequence  $r^{\alpha}_k$ in $\mathfrak{R}^{\alpha}_{mn}(2\pi)$ such that  $$\|f- r^{\alpha}_k\|_{\infty} < d+ \frac{1}{k}, ~ ~k=1,2,\dots .$$ It follows that   $$ \|r^{\alpha}_k\|_{\infty} \leq \|r^{\alpha}_k - f\|_{\infty} + \| f\|_{\infty} \leq d+1+\| f\|_{\infty},~ ~ k=1,2,\dots .$$
Let $ r^{\alpha}_k = \mathcal{F}_{\Delta,L}^{\alpha}(r_k),$ where $r_k=\dfrac{p_k}{q_k}$, $ p_k \in \mathfrak{T}_m$, $q_k \in \mathfrak{T}_n$, $\|q_k\|_{\infty}=1$ and $q_k(x) > 0$ on $I .$ Since the fractal operator is bounded below, there exists $C>0$ such that $$ C \|f\|_{\infty} \leq \|\mathcal{F}_{\Delta,L}^{\alpha}(f)\|_{\infty} ~~  \forall~~ f \in \mathcal{C}(I).$$ Therefore
$$ \|r_k\|_{\infty} \leq \frac{1}{C}\|\mathcal{F}_{\Delta,L}^{\alpha}(r_k)\|_{\infty}= \frac{1}{C}\|r^{\alpha}_k\|_{\infty} \leq \frac{1}{C}\big( d+1+\|f\|_{\infty}\big):= K.$$
Since $\mathfrak{T}_m, \mathfrak{T}_n$ are finite dimensional spaces and $$|p_k(x)|=|q_k(x)||r_k(x)| \leq \|q_k\|_{\infty}\|r_k\|_{\infty} \leq A,$$ the pairs $(p_k,q_k)$ lie in the compact sets  defined by the inequalities $\|p\|_{\infty} \leq A$ and $\|q\|_{\infty} =1 .$ We may assume, by passing to a subsequence if necessary, that $p_k \rightarrow p$ and $q_k \rightarrow q.$ Clearly, $\|q\|_{\infty}=1;$ whence using the Haar condition there can be at most $2n$ zeros for $q$.  At the points that are not zeros of $q,$ $\dfrac{p(x)}{q(x)}$ is well defined, and we have $\dfrac{p_k(x)}{q_k(x)}\rightarrow \dfrac{p(x)}{q(x)}.$ Therefore at points in $I$ where $q$ does not vanish,
$$\dfrac{|p(x)|}{|q(x)|} \leq A, \quad |p(x)| \leq A|q(x)|.$$ Since there are only finite number of zeros for $q,$ by continuity, the last inequality holds for all $x \in I.$  One can  apply previous lemma (perhaps repeatedly) to obtain other trigonometric polynomials $p_*$ and $q_*$ such that $\deg(p_*) < \deg(p)$ , $\deg(q_*) < \deg(q)$,  $q_*(x)>0$  on $ I$ and $p_*(x)q(x)= p(x)q_*(x).$ The resulting element $r_*:=\dfrac{p_*}{q_*}$ is in $\mathfrak{R}_{mn}(2 \pi).$ As $ r_k \rightarrow r_*$ uniformly and $ \mathcal{F}_{\Delta,L}^{\alpha}$ is a bounded linear map, we get $ r_k^{\alpha} \rightarrow r_*^{\alpha}$ and hence  $f-r_k^{\alpha} \rightarrow f-r_*^{\alpha}.$ By the continuity of norm, we have,  $\|f-r_k^{\alpha}\|_{\infty} \rightarrow \|f-r_*^{\alpha}\|_{\infty}.$ Therefore, $\|f-r_*^{\alpha}\|_{\infty}=d.$
\end{proof}
\begin{remark}
Approach in the previous proof  is identical to the one used in \textup{\cite[Theorem 4.1]{Rational}} for proving the proximality of the class of fractal rational functions in $\mathcal{C}(I)$ except for a few lines at the end. However, we included a expanded rendition of the arguments for the sake of completeness and record.
\end{remark}

\begin{remark}
 It is known that (Cf. Theorem \ref{prelthm}) for $|\alpha |_{\infty}< \|L\|^{-1}$, the fractal operator $\mathcal{F}_{\Delta,L}^{\alpha}$ is bounded below. Therefore, for $|\alpha |_{\infty}< \|L\|^{-1}$,  $\mathfrak{R}^{\alpha}_{mn}(2\pi) $ is a proximinal approximately compact subset of $\mathcal{C}(2\pi)$.
\end{remark}
In general,  best approximant from $\mathfrak{R}^{\alpha}_{mn}(2\pi)$ to $f \in \mathcal{C}(2\pi)$ may not be unique. For $f \in \mathcal{C}(2\pi),$ let us write
$$P_{\mathfrak{R}^{\alpha}_{mn}(2\pi)}(f)=   \Big\{r^\alpha \in \mathfrak{R}^{\alpha}_{mn}(2\pi): \|f-r^\alpha\|_\infty= \text{dist}\big(f, \mathfrak{R}^{\alpha}_{mn}(2\pi)\big)   \Big\}.$$
\begin{theorem}
If the fractal operator $\mathcal{F}_{\Delta,L}^{\alpha}$ is bounded below, then the set-valued map $P_{\mathfrak{R}^{\alpha}_{mn}(2\pi)}: \mathcal{C}(2\pi) \rightrightarrows \mathfrak{R}^{\alpha}_{mn}(2\pi)$ supported on the nonempty proximal subset $\mathfrak{R}^{\alpha}_{mn}(2\pi)$ is upper semicontinuous and closed.
\end{theorem}
\begin{proof}
By Theorem \ref{proximal},  $\mathfrak{R}^{\alpha}_{mn}(2\pi)$ is a nonempty approximately compact subset of the normed linear space $\mathcal{C}(2\pi)$.
Therefore the multi-valued map $P_{\mathfrak{R}^{\alpha}_{mn}(2\pi)}: \mathcal{C}(2\pi) \rightrightarrows \mathfrak{R}^{\alpha}_{mn}(2\pi)$ is upper semicontinuous and its values are compact. This follows by a result that if $X$ is a normed linear space and $V$ is a nonempty approximately compact subset of $X$, then the metric projection set-valued function $P_V : X \rightrightarrows V$ is upper semicontinuous and its values are compact (see, for instance, \textup{\cite[p. 440]{DVpai}}). The set-valued map $P_{\mathfrak{R}^{\alpha}_{mn}(2\pi)} $ is closed follows from the fact that if $X$ is a topological space, $Y$ is a Hausdorff space and $T:X \rightrightarrows Y$ is upper semicontinuous with compact values, then $T$ is closed (see, for instance, \textup{\cite[p. 434]{DVpai}}).
 \end{proof}
\begin{remark}
Taking $n=0$, $ \mathfrak{R}^{\alpha}_{m0}(2\pi)= \mathfrak{T}^{\alpha}_m(2\pi) .$ Therefore the above theorems and remark are valid for the class of fractal trigonometric polynomials as well. This observation serves as  an addendum to the researches in \cite{MAN3}.
\end{remark}
\subsection{Approximation Error Bound}\label{AEB}
Define $$ E^{\alpha}_{mn}(f;[-\pi,\pi]):=\inf\big\{\|f - r^{\alpha} \|_{\infty}:r^{\alpha} \in \mathfrak{R}^{\alpha}_{mn}(2\pi) \big\}$$ and $$ E_{mn}(f;[-\pi,\pi]):=\inf\big\{\|f - r \|_{\infty}:r \in \mathfrak{R}_{mn}(2\pi)\big \}.$$

\begin{theorem} \label{error1a}
Let $f \in \mathcal{C}( 2 \pi)$. Then,
$$ E^{\alpha}_{mn}(f;[-\pi,\pi])\leq \frac{1+|\alpha|_\infty(\|Id-L\|-1)}{1-|\alpha|_\infty} E_{mn}(f;[-\pi,\pi])+ \frac{|\alpha |_{\infty}}{1-|\alpha |_{\infty}}\|Id-L\|~\|f\|_\infty.$$
\end{theorem}
\begin{proof}
Let $f \in \mathcal{C}(2 \pi)$ and $r_*$ be a best approximant to $f$ from $\mathfrak{R}_{mn}(2 \pi)$. We have

\begin{equation*}
  \begin{aligned}
   E^{\alpha}_{mn}(f;[-\pi,\pi]) & \leq \|f - r^{\alpha}_* \|_{\infty}\\
& \leq \|f - r_* \|_{\infty} + \|r_* - r^{\alpha}_* \|_{\infty} \\
& = E_{mn}(f;[-\pi,\pi])+\|r_* - r^{\alpha}_* \|_{\infty}\\
& \leq E_{mn}(f;[-\pi,\pi])+\frac{|\alpha |_{\infty}}{1-|\alpha |_{\infty}}\|Id-L\|~\|r_*\|_\infty\\
& \leq E_{mn}(f;[-\pi,\pi])+\frac{|\alpha |_{\infty}}{1-|\alpha |_{\infty}}\|Id-L\|~ \big(\|f-r^*\|_\infty+\|f\|_\infty \big),
\end{aligned}
\end{equation*}
and hence the theorem.
\end{proof}
Let $n \in \mathbb{N}$ and $x_{kn}= \frac{2k\pi}{n}$, $k=0,1,\dots,n-1$. Let $f$ be an arbitrary $2\pi$-periodic continuous function. From
 \cite{AKVarma}, we recall a sequence of positive linear interpolating operators $\wedge_n$, $n=1,2,\dots$, which map $\mathcal{C}(2\pi)$ into the set of rational trigonometric functions of order $ \leq 2n-2$ defined by
 $$\wedge_n(f,x)= \frac{\sum_{k=0}^{n-1}f(x_{kn})J^2_n(x-x_{kn})}{1-((n^2 -1)/3n^2)(1-\cos(nx))},$$
where $J_n$ are Jackson functions
$$J_n(x)= \Big(  \frac{\sin(nx/2)}{n~\sin(x/2)}\Big)^2.$$
Let us recall also that the modulus of continuity of a bounded function $f$ on the compact interval $I$ is defined by
$$\omega_f (\delta) = \omega_f \big(I; \delta  \big):= \sup\Big\{|f(x)-f(y)|: x, y \in I, |x-y| \le \delta  \Big\}.$$
 \begin{theorem}(\cite{AKVarma}).\label{error2}
Let $f \in \mathcal{C}(2\pi)$. Then, for $n=2,3,4,\dots ,$ and for all $x,$ $$ |f(x)- \wedge_n(f,x)| \leq 2\omega_f\Big( \frac{\pi \sqrt{3}}{n}\Big).$$
\end{theorem}
Theorems \ref{error1a}-\ref{error2} now dictate
\begin{theorem}
Let $f \in \mathcal{C}(2\pi)$ with modulus of continuity $ \omega_f(\delta)$. Then, $$ E^{\alpha}_{nn}(f;[-\pi,\pi])\leq \frac{1+|\alpha|_\infty(\|Id-L\|-1)}{1-|\alpha|_\infty} 2\omega_f\Big( \frac{2 \pi \sqrt{3}}{n+2}\Big) + \frac{|\alpha |_{\infty}}{1-|\alpha |_{\infty}}\|Id-L\|~\|f\|_\infty.$$
\end{theorem}

\section{Some Comments and Corrections}
This section aims to provide corrections and comments to some results scattered in the literature that are based on the concept of $\alpha$-fractal functions.
\subsection{On minimax error}\label{SCC1}
Let us begin by noting that a result similar to Theorem \ref{error1a} in the previous section is announced in \textup{\cite[Theorem 4.3]{Rational}} to compare fractal rational minimax error with classical rational minimax error. With the notation
\begin{equation*}
\begin{split}
 \mathcal{R}_{mn}(I) :=&~ \Big \{r=\frac{p}{q}: p \in \mathcal{P}_m(I), q \in \mathcal{P}_n(I); q >0 ~\text{on}~ I \Big \}, \\ \mathcal{R}_{mn}^\alpha (I) :=&~ \mathcal{F}_{\Delta, L}^\alpha \big ( \mathcal{R}_{mn} (I)\big),
 \end{split}
 \end{equation*}
where $\mathcal{P}_k(I)$ is the space of all algebraic polynomials of degree at most $k$, the authors claim that
\begin{theorem}\textup{\cite[Theorem 4.3]{Rational}} \label{Thmamend1}.
For any $f \in \mathcal{C}(I )$, $$\text{dist}\big(f,\mathcal{R}_{mn}^\alpha(I)\big) \le \text{dist}\big(f,\mathcal{R}_{mn}(I)\big).$$
\end{theorem}
However, the proof of the above theorem as mentioned in \cite{Rational}  is inaccurate. The inaccuracy comes from the fact that the proof  uses:
\begin{equation*}
\begin{split}
& \inf \{ \|f-\mathcal{F}^\alpha(r)\|_\infty: r \in \mathcal{R}_{mn}(I)\}\\ \le &~ \inf\big\{\|f-r\|_\infty  + \|r- \mathcal{F}^\alpha(r)\|_\infty: r \in \mathcal{R}_{mn}(I) \big\}\\
\le &~ \inf \big\{ \|f-r\|_\infty: r \in \mathcal{R}_{mn}(I)\big \} +  \inf\big\{ \|r-\mathcal{F}^\alpha(r)\|_\infty: r \in \mathcal{R}_{mn}(I)\big \},
\end{split}
\end{equation*}
 which is not true. If $A=\{x_\beta: \beta \in \Lambda\}$, $B=\{y_\beta: \beta \in \Lambda\}$ are subsets of $\mathbb{R}$ and $C=\{x_\beta+y_\beta: \beta \in \Lambda\}$, then it is easy to see that  $\inf C \ge \inf A + \inf B$. However, in general, $\inf C \le \inf A + \inf B$ is not true. Let us recall also that, in fact, $\inf(A+B) = \inf(A) + \inf(B)$ holds, however here $C \neq A+B$. P. Viswanathan regrets for this careless mistake and would like to mention that it was also observed and pointed out by Prof. Navascu\'{e}s in a different context during some personal communications.
Our result in Theorem  \ref{error1a} suggests that  at this point Theorems \ref{Thmamend1} above can be corrected by supplying suitable additional terms. For instance, with notation as in \cite{Rational}, we have
\begin{theorem} \label{Thmamend3}
For any $f \in \mathcal{C}(I )$,
\begin{equation*}
\begin{split}
\text{dist}\big(f,\mathcal{R}_{mn}^\alpha(I)\big) \le &~ \text{dist}\big(f,\mathcal{R}_{mn}(I)\big) \\ &~+ \frac{|\alpha|_\infty}{1-|\alpha|_\infty} \|Id-L\| \Big(\text{dist}\big(f,\mathcal{R}_{mn}(I)\big) + \|f\|_\infty\Big).
\end{split}
\end{equation*}
\end{theorem}
The same incorrect arguments in the proof of Theorem \ref{Thmamend1} is repeated  recently for the class of  Bernstein $\alpha$-fractal rational functions in \cite{Vij2}. We note that the theorem remains valid and supply a correct proof for it.  Let us recall the following notation as in \cite{Vij2}. For a fixed partition $\Delta$ and scale vector $\alpha$
\begin{equation*}
\begin{split}
\mathcal{R}_{l,m}(I) :=&~ \Big \{r=\frac{p}{q}: p \in \mathcal{P}_l(I), q \in \mathcal{P}_m(I); q >0 ~\text{on}~ I \Big \},\\ \mathcal{R}_{l,m}^\alpha (I) :=&~ \big\{ \mathcal{F}^\alpha_{\Delta, B_n}(r): r \in \mathcal{R}_{l,m}(I), n \in \mathbb{N}\big\}.
\end{split}
\end{equation*}

\begin{theorem}\textup{\cite[Theorem 3.9]{Vij2}}. \label{Thmamend2}
For any $f \in \mathcal{C}(I ),$
$$\text{dist}\big(f,\mathcal{R}_{l,m}^\alpha(I)\big) \le \text{dist}\big(f,\mathcal{R}_{l,m}(I)\big).$$
\end{theorem}
\begin{proof}
Let $r^*$ be the unique best approximant to $f \in \mathcal{C}(I )$ from $\mathcal{R}_{l,m}(I)$, that is, $\|f-r^*\|_\infty = \text{dist}\big(f,\mathcal{R}_{l,m}(I)\big)$ (see, for instance, \textup{\cite[p. 164]{Cheney}}).
 Using item (1) in Theorem \ref{prelthm} we have
\begin{equation*}
\begin{split}
\text{dist}\big(f,\mathcal{R}_{l,m}^\alpha(I)\big) \le &~ \|f- \mathcal{F}^\alpha_{\Delta, B_n}(r^*)\|_\infty \\
\le &~ \|f-r^*\|_\infty + \|r^*-\mathcal{F}^\alpha_{\Delta, B_n}(r^*)\|_\infty\\
\le &~ \text{dist}\big(f,\mathcal{R}_{l,m}(I)\big)+ \frac{|\alpha|_\infty}{1-|\alpha|_\infty}\|Id- B_n\| \|r^*\|_\infty
\end{split}
\end{equation*}
Since the above estimate holds for all $n \in \mathbb{N}$ and $\|Id-B_n\| \to 0$ as $n \to \infty$, we infer that
$\text{dist}\big(f,\mathcal{R}_{l,m}^\alpha(I)\big) \le \text{dist}\big(f,\mathcal{R}_{l,m}(I)\big)$ and thus the proof.
\end{proof}
\begin{remark}
In view of Theorems \ref{Thmamend3} -\ref{Thmamend2} it appears that the approximation class  $\mathcal{R}_{m,n}^\alpha(I)$
of Bernstein fractal rational functions introduced in \cite{Vij2} is ``better" than the class $\mathcal{R}_{mn}^\alpha(I)$ of fractal rational functions that made its debut in \cite{Rational}. In this regard, let us note that corresponding to a fixed  rational function $r$ of order $(m,n)$, there exists a unique fractal rational function  $r_{\Delta,L}^\alpha$ whereas there exist a sequence of Bernstein fractal rational functions $\big(r_{\Delta,B_n}^\alpha\big)$ converging to $r$.

\end{remark}

\subsection{On Bernstein $\alpha$-fractal functions}
As mentioned in the introductory section, Inequality (\ref{eq1}) should convince the reader that the fractal function  $f_{\Delta,b}^\alpha$ can be made close to the seed function $f$ by taking the parameter map $b$ close to $f$.  In \cite{Vij1,Vij2,Vij3}, the author uses this simple observation effectively  by selecting $b=B_n(f)$, the Bernstein polynomials for $f$ to introduce what is called Bernstein $\alpha$-fractal functions.  In particular, using a result by Akhtar et. al. (see Theorem \ref{Akhtar}), the following claim is made in \cite{Vij1,Vij2,Vij3}.
\begin{theorem} \textup{\cite[Theorem 2]{Vij1}},\textup{\cite[Theorem 2.3]{Vij2}}, \textup{\cite[Theorem 2]{Vij3}}.
Let $f \in \mathcal{C}(I).$ Let $\Delta=\{x_0,x_1,\dots , x_N\}$ be a partition of $I=[x_0,x_N]$ satisfying $x_0<x_1< \dots < x_N$ and $ \alpha =(\alpha_1,\alpha_2, \dots ,\alpha_{N}) \in (-1,1)^N.$ If the $\alpha$-fractal functions in the sequence $\big(f^{\alpha}_{\Delta,B_n}\big)_{n=1}^{\infty}$ are obtained with same fixed choice of scaling vector $\alpha$ whose components  satisfy the condition $\sum_{i=1}^{N} |\alpha_i | > 1,$ then all the $\alpha$-fractal functions in the sequence $\big(f_{\Delta,B_n}^{\alpha}\big)_{n=1}^{\infty}$ have the same fractal dimension $D \in (1,2)$ and $\lim_{n \to \infty}f_{\Delta, B_n}^{\alpha}=f.$
\end{theorem}
 Theorem \ref{Akhtar} has the hypothesis that the seed function $f$ and base function $b$ are Lipschitz continuous and data points sampled from $f$ are not collinear. The Bernstein $\alpha$-fractal functions use Bernstein polynomials as base function which obviously are Lipschitz.  However, to apply Theorem \ref{Akhtar} other hypotheses are to be taken care, and a possible refinement to the above theorem could be
\begin{theorem}
Let $f: I \to \mathbb{R}$ be a  Lipschitz continuous function and $\Delta=\{x_0,x_1,\dots , x_N: x_0<x_1< \dots < x_N\}$ be a partition of $I=[x_0,x_N]$ such that the data set $\big\{\big(x_i,f(x_i)\big): i=0,1,\dots,N\big\}$ is not collinear. Let $ \alpha =(\alpha_1,\alpha_2, \dots ,\alpha_{N}) \in (-1,1)^N$ be a fixed vector such that $\sum_{i=1}^{N} |\alpha_i | > 1.$ Then the graphs of the Bernstein $\alpha$-fractal functions $f_{\Delta, B_n}^{\alpha}$, $n \in \mathbb{N}$ have the same box  dimension
 $D$ given by the formula in Theorem \ref{Akhtar} and  the sequence $ \Big(f_{\Delta, B_n}^{\alpha}\Big)_{n=1}^{\infty}$ converges to $f$ uniformly.
\end{theorem}
Now we shall prove that the assumption of Lipschitz continuity of $f$ can be dropped, if one intends only to obtain a sequence of $\alpha$-fractal functions converging  uniformly to $f$ with additional property that the graphs of all functions in this sequence have the same box dimension.
\begin{theorem}
Let $f \in \mathcal{C}(I)$ and $\Delta=\{x_0,x_1,\dots , x_N: x_0<x_1< \dots < x_N\}$ be a partition of $I=[x_0,x_N]$. Let $ \alpha =(\alpha_1,\alpha_2, \dots ,\alpha_{N}) \in (-1,1)^N$ be a fixed vector. Then there exists a sequence of $\alpha$-fractal functions with graphs having same box dimension that converges  uniformly to $f$.
\end{theorem}
\begin{proof}
  Let $f \in \mathcal{C}(I).$  Without loss of generality assume that the points in $\big\{\big(x_i,f(x_i)\big): i=0,1,\dots,N\big\}$ obtained by sampling $f$ are not collinear. Consider the sequence of Bernstein polynomials $(p_n)_{n \in \mathbb{N}}$ that converges to $f$ uniformly. That is, $p_n = B_n(f)$ for $n \in \mathbb{N}$, where $B_n$ is the $n$-th Bernstein operator. Fix a partition $\Delta$ and a scale vector $\alpha$. For each fixed $n \in \mathbb{N}$, find the $\alpha$-fractal function $(p_n)_{\Delta, B_n(p_n)}^\alpha$ corresponding to $p_n$ by choosing the base function as $B_n(p_n)$. We have
 \begin{equation*}
 \begin{split}
 \| f- (p_n)_{\Delta, B_n(p_n)}^\alpha\| _\infty \le &~ \|f-p_n \|_\infty + \|p_n - (p_n)_{\Delta, B_n(p_n)}^\alpha\|_\infty \\
 \le &~ \|f -p_n \|_\infty + \frac{|\alpha|_\infty}{1-|\alpha|_\infty} \| p_n - B_n(p_n)\|_\infty\\
 \le &~  \|f -p_n \|_\infty + \frac{|\alpha|_\infty}{1-|\alpha|_\infty} \|Id-B_n \| \|p_n \|_\infty \\
 = &~ \|f-p_n\|_\infty + \frac{|\alpha|_\infty}{1-|\alpha|_\infty} \|Id-B_n \| \|B_n(f)\|_\infty \\
 \le &~ \|f-p_n\|_\infty + \frac{|\alpha|_\infty}{1-|\alpha|_\infty} \|Id-B_n \| \|f\|_\infty.
 \end{split}
 \end{equation*}
 The above estimate shows that $(p_n)_{\Delta, B_n(p_n)}^\alpha \to f$ uniformly  as $ n \to \infty.$ For each fixed $n \in \mathbb{N}$, the germ function $p_n$ and base function $B_n(p_n)$ are Lipschitz and the set of data points $\{(x_i, p_n(x_i): I=0,1,\dots, N\}$ is not collinear. Therefore, by the formula in Theorem \ref{Akhtar}, the box dimensions of the graphs of $(p_n)_{\Delta, B_n(p_n)}^\alpha$, which depend only on the partition and scaling vector, are all same.
 \end{proof}
 \subsection{On non-self-referential Bernstein fractal rational functions}\label{nonselfssec} The class of fractal rational functions studied in detail in \cite{Rational} is defined as the image of the set of rational functions under the bounded linear map $\mathcal{F}_{\Delta, L}^\alpha$. However, it is hinted in \textup{\cite[Remark 3.2]{Rational}} that a class of fractal rational functions can also be defined by considering the quotients of suitable fractal polynomials.
  Let us recall these  two classes of fractal rational functions. Following notation in \cite{Rational},  let $\mathcal{P}_k (I)$ be the set of polynomials of degree less than or equal to $k$ defined on $I$,
  \begin{equation*}
  \begin{split}
 &\mathcal{R}_{mn}(I) :=\Big\{\frac{p}{q}: p \in \mathcal{P}_m(I), q \in \mathcal{P}_n(I);~~ q>0~~ \text{on}~~I  \Big\},\\
  & \mathcal{P}_k ^\alpha (I) := \mathcal{F}_{\Delta, L}^\alpha\big(\mathcal{P}_k (I)\big),\\  & \mathcal{R}_{mn}^\alpha(I):= \mathcal{F}_{\Delta, L}^\alpha \big(\mathcal{R}_{mn}(I)\big),\\
& \mathcal{K}_{mn}^\alpha(I):= \Big\{\frac{p^\alpha}{q^\alpha}: p \in \mathcal{P}_m^\alpha(I), q \in \mathcal{P}_n^\alpha(I);~~ q^\alpha>0~~ \text{on}~~I  \Big\}.
 \end{split}
 \end{equation*}
 In \cite{Vij2}, by taking $L$ as the sequence of Bernstein operators, two classes of Bernstein  rational functions are considered, which we shall denote again by $\mathcal{R}_{mn}^\alpha(I)$ and $\mathcal{K}_{mn}^\alpha(I)$. The author claims that an element $\varphi=\dfrac{p^\alpha}{q^\alpha}\in  \mathcal{K}_{mn}^\alpha(I)$ is \emph{non-self-referential} function as its graph does not correspond to any IFS \textup{\cite[Section 4]{Vij2}}.  To this end, we observe the following.
 \par We know  \big(Cf. item (4) Theorem \ref{prelthm}\big) that for the scaling vector $\alpha \in (-1,1)^{N}$ satisfying $|\alpha|_\infty < \big(1 + \|Id - L\|\big)^{-1}$, the operator $\mathcal{F}^\alpha: \mathcal{C}(I) \to \mathcal{C}(I)$ is invertible. Consequently, for $\alpha$ satisfying the aforementioned condition and $\varphi := \dfrac{p^\alpha}{q^\alpha} \in \mathcal{K}_{mn}^\alpha(I)\subseteq \mathcal{C}(I)$, there exists
$g \in \mathcal{C}(I)$ such that $\varphi = \mathcal{F}^\alpha (g)=g^\alpha$. By the very construction of the $\alpha$-fractal function, the graph of $g^\alpha$, and consequently the graph of $\varphi$, is the attractor of an IFS. In fact, $\varphi$ satisfies the self-referential equation

$$ g_{\Delta,b}^\alpha(x) = g(x) + \alpha_i (g_{\Delta,b}^{\alpha}- b)\big(L_i^{-1}(x)\big) ~ ~ ~~\forall~~ x \in I_i,~~ i \in \mathbb{N}_N.$$
Therefore, it is perhaps an abuse of terminology to regard functions in $\mathcal{K}_{mn}^\alpha(I)$ as non-self-referential for all permissible choice of $\alpha$. Let us remark that the question whether $\varphi \in \mathcal{K}_{mn}^\alpha(I)$ is equal to $\mathcal{F}^\alpha (r)$ for some rational function $r$ remains unsettled.
\subsection{On multi-valued fractal operator}
 The definition of Bernstein $\alpha$-fractal function $f_{\Delta, B_n(f)}^\alpha$ corresponding to each $f \in \mathcal{C}(I)$ provides
 \begin{enumerate}
 \item  a sequence of single-valued operators $(\mathcal{F}_{\Delta, B_n}^\alpha)_{n \in \mathbb{N}}$ defined by $$\mathcal{F}_{\Delta, B_n}^\alpha: \mathcal{C}(I) \to \mathcal{C}(I); \quad \mathcal{F}_{\Delta, B_n}^\alpha(f)= f_{\Delta, B_n(f)}^\alpha~~\text{for ~each}~~n \in \mathbb{N}.$$
 \item a multi-valued operator $\mathcal{F}^\alpha : \mathcal{C}(I) \rightrightarrows \mathcal{C}(I)$ defined by
 $$ \mathcal{F}^\alpha (f) = \{f_{\Delta, B_n(f)}^\alpha: n \in \mathbb{N}\}.$$
 \end{enumerate}
It appears that in \cite{Vij3}, the author switches between the sequence of single-valued operators and multi-valued operator in items (1) -(2) above without making a clear distinction between them. For instance, it is claimed that
\begin{theorem}\textup{\cite[Theorem 3]{Vij3}}. \label{thmmulval}
Let $\mathcal{C}(I)$ be endowed with the supnorm. The multi $\alpha$-operator $\mathcal{F}^\alpha: \mathcal{C}(I) \rightrightarrows \mathcal{C}(I)$ defined by
$\mathcal{F}^\alpha (f) = f_{\Delta, B_n(f)}^\alpha= f_n^\alpha$ is linear and bounded.
\end{theorem}
 A careful examination of the proof of the above theorem indicates that, the author treats $\mathcal{F}^\alpha$ as an ordinary function, but not as a set-valued map. In what follows, we shed some light on this. Since for $\alpha =0$, $\mathcal{F}_{\Delta, B_n}^\alpha(f) =f$ for all $f \in \mathcal{C}(I)$ and $n \in \mathbb{N}$, we consider the case $\alpha \neq 0.$
 \par
Since the Bernstein operator $B_n: \mathcal{C}(I) \to \mathcal{C}(I)$, $n \in \mathbb{N}$ is a bounded linear operator, by Theorem \ref{prelthm}, it is straightforward to see that $\mathcal{F}_{\Delta, B_n}^\alpha$ is a bounded linear operator for each $n \in \mathbb{N}$. However, as mentioned in the introductory section the linearity and boundedness of the multi-valued operator is dealt with a slightly  different approach \cite{Aubin}. A suitable question to ask would be whether the multi-valued operator  $\mathcal{F}^\alpha$ is a closed convex process. We provide a partial answer to this question and arguments to conclude that $\mathcal{F}^\alpha$ is not ``linear", contradicting Theorem \ref{thmmulval}. As a prelude, let us recall a few definitions and results.
\begin{definition}(\cite{Aubin}).
  Let $X$ and $Y$ be two real  normed linear spaces over $\mathbb{R}$. For a set-valued map $T$ from $X$ to $Y$ denoted by  $T: X \rightrightarrows Y$, the
   domain of $T$ is defined by
  $\text{Dom}(T):= \{x \in X: T(x) \neq \emptyset\}.$ Then $T: X \rightrightarrows Y$ is
  \begin{enumerate}
  \item  \emph{convex}  if for all $x_1, x_2 \in \text{Dom}(T)$ and for all $\lambda \in [0,1],$  $$\lambda T(x_1)+(1-\lambda)T(x_2) \subseteq T\big(\lambda x_1+(1-\lambda)x_2\big).$$

  \item  \emph{process}  if  for all $~x \in X$  and for all  $\lambda > 0$, $$\lambda T(x)=  T(\lambda x) ~\text{and}~ 0 \in T(0).$$

      \item \emph{linear} if  for all $x_1, x_2 \in \text{Dom}(T)$ and for all $\beta, \gamma \in \mathbb{R},$ $$\beta T(x_1)+ \gamma T(x_2) \subseteq T\big(\beta x_1+\gamma x_2\big).$$

          \item \emph{closed} if graph of $T$ $$G_T:= \big\{(x,y)\in X \times Y: y \in T(x)  \big\}$$  is closed.

          \item \emph{Lipschitz} if there exists a constant $l >0$ such that  for all $x_1, x_2 \in \text{Dom}(T)$

          $$T(x_1) \subseteq T(x_2) + l \|x_1-x_2\| U_Y,$$
          where $U_Y$ is the closed unit ball in $Y$.

  \end{enumerate}
  \end{definition}
  \begin{theorem}\textup{\cite[Corollary 1.4]{DS}}. \label{Multhm2}
 Let $X$ and $Y$ be real vector spaces and $\mathcal{P}_0(Y)$ be the collection of all nonempty subsets of $Y$. If a set-valued map $T: X \to \mathcal{P}_0(Y)$ is linear and $T(0)= \{0\}$, then $T$ is single-valued.
 \end{theorem}
  \begin{theorem}\textup{\cite[Corollary 2.1]{DS}}.\label{Multhm2a}
  Let $X$ and $Y$ be real vector spaces and $\mathcal{P}_0(Y)$ be the collection of all nonempty subset of $Y$. If a set-valued map $T: X \to \mathcal{P}_0(Y)$ is such that $T(x_0)$ is singleton for some $x_0\in X$, then $T: X \to \mathcal{P}_0(Y)$ is convex if and only if $T$ is single-valued and affine.
  \end{theorem}
  \begin{theorem}\label{Multhm3}
  The set-valued map $: \mathcal{F}^\alpha: \mathcal{C}(I) \rightrightarrows \mathcal{C}(I)$ defined by $$\mathcal{F}^\alpha(f) =\{f^{\alpha}_{\Delta,B_n(f)}: n \in \mathbb{N} \}$$ is a Lipschitz process.
  \end{theorem}
  \begin{proof}
   Let $f \in \mathcal{C}(I)$ and $\lambda > 0.$  Since for each fixed $n \in \mathbb{N}$, the operator defined by $\mathcal{F}_{\Delta, B_n}^\alpha(f)= f_{\Delta, B_n(f)}^\alpha$ is linear

   \begin{equation*}
                   \begin{aligned}
            \mathcal{F}^\alpha (\lambda f) =  &\{(\lambda f)^{\alpha}_{\Delta,B_n}:n \in \mathbb{N} \}\\
                 =&\{ \lambda f^{\alpha}_{\Delta,B_n}:n \in \mathbb{N}\}\\
                  = &\lambda\{f^{\alpha}_{\triangle,B_n}:n \in \mathbb{N}  \}\\
                   =&  \lambda \mathcal{F}^\alpha(f).
\end{aligned}
   \end{equation*}
Further, since for each $n \in \mathbb{N}$, $\mathcal{F}_{\Delta, B_n}^\alpha$ is a linear operator, it follows that for
$\mathcal{F}_{\Delta, B_n}^\alpha(0) = 0_{\Delta, B_n(0)}^\alpha=0$. Thus, $\mathcal{F}^\alpha (0) = \{0\},$ and consequently  $\mathcal{F}^\alpha$ is a process.  Let $f,g \in \mathcal{C}(I).$
 Using the  functional equation for the $\alpha$-fractal function we have
        $$ f^{\alpha}_{\Delta,B_m}(x)= f(x)+\alpha_i \big(f^{\alpha}_{\Delta,B_m}- B_m(f)\big)\circ L_i^{-1}(x)  ~~\forall~~ x \in I_i,~~ i \in \mathbb{N}_N,$$
        and
 $$ g^{\alpha}_{\Delta,B_m}(x)= g(x)+\alpha_i \big(g^{\alpha}_{\Delta,B_m}- B_m(g)\big)\circ L_i^{-1}(x) ~~\forall~~ x~~ \in I_i,~~ i \in \mathbb{N}_N.$$
 Therefore,
 \begin{equation*}
           \begin{aligned}
                    f^{\alpha}_{\Delta,B_m}(x)- g^{\alpha}_{\Delta,B_m}(x) =& f(x)- g(x)+\alpha_i(f^{\alpha}_{\Delta,B_m}- g^{\alpha}_{\Delta,B_m})\circ L_i^{-1}(x)\\&+\alpha_i \big(B_m(g)- B_m(f)\big)\circ L_i^{-1}(x),
          \end{aligned}
    \end{equation*}
 for all $x \in I_i,~~ i \in \mathbb{N}_N.$
 Further, we deduce
 $$\|f^{\alpha}_{\Delta,B_m}- g^{\alpha}_{\Delta,B_m}\|_{\infty} \le \frac{1+|\alpha|_{\infty}\|B_m\|}{1- |\alpha|_{\infty}} \|f-g\|_{\infty}.$$
 Since $\|B_m \| \le 1 ,$ we get
  $$\|f^{\alpha}_{\Delta,B_m}- g^{\alpha}_{\Delta,B_m}\|_{\infty} \le \frac{1+|\alpha|_{\infty}}{1- |\alpha|_{\infty}} \|f-g\|_{\infty}.$$
  Choosing $l= \dfrac{1+ |\alpha|_{\infty}}{1-|\alpha|_{\infty}}, $ we have
  $$ \mathcal{F}^\alpha(g) \subseteq \mathcal{F}^\alpha(f) +l~ \|f-g\|_{\infty}U_{\mathcal{C}(I)},$$
  establishing that $\mathcal{F}^\alpha$ is a Lipschitz set-valued map.
\end{proof}
\begin{remark}
For the set-valued map $: \mathcal{F}^\alpha: \mathcal{C}(I) \rightrightarrows \mathcal{C}(I)$ defined by $\mathcal{F}^\alpha(f) =\{f^{\alpha}_{\Delta,B_n(f)}: n \in \mathbb{N} \}$, we have the following.
\begin{enumerate}
\item as observed in the previous theorem, $\mathcal{F}^\alpha (0) = \{0\}.$

\item $ \mathcal{F}^\alpha: \mathcal{C}(I) \rightrightarrows \mathcal{C}(I)$ is not single-valued. This follows by observing that for a scale vector $\alpha \neq 0$, $b \neq c$ implies $f_{\Delta,b}^\alpha \neq f_{\Delta,c}^\alpha$.
\end{enumerate}
\end{remark}
In view of the previous remark, Theorems \ref{Multhm2}-\ref{Multhm2a} provide
\begin{theorem}
The multi-valued operator $ \mathcal{F}^\alpha: \mathcal{C}(I) \rightrightarrows \mathcal{C}(I)$  is not convex, and hence, in particular, not linear.
\end{theorem}
Similarly, the following theorem allegedly reports that the multi-valued operator $\mathcal{F}^\alpha: \mathcal{C}(I) \rightrightarrows \mathcal{C}(I)$ is ``bounded below" and ``non-compact" whereas in the proof author deals actually with the single-valued operator $\mathcal{F}_{\Delta,B_n}^\alpha$ for each fixed $n \in \mathbb{N}$.
\begin{theorem}\textup{\cite[Theorem 4]{Vij3}}. Let $\theta=\max \{\|B_1\|, \|B_2\|,\dots ,\|B_{N_0}\|, 1+\epsilon\}.$ If $|\alpha|_{\infty} < \frac{1}{\theta},$ then $ \mathcal{F}^{\alpha}$ is bounded below and not compact.
 \end{theorem}
Our search for the set-valued analogues  of the property of being bounded below and compactness of a linear operator came up emptyhanded.
Let us further remark that in the case where one intends to work only with the single-valued fractal operator $\mathcal{F}_{\Delta,B_n}^\alpha$, by Theorem \ref{prelthm}, it follows that it is bounded below and not compact for any choice of the scaling vector $\alpha$ with $|\alpha|_\infty<1$. This is because of the fact that since $\|B_n\| \le 1$, the condition  $|\alpha|_\infty < \|B_n\|^{-1}$ is automatically satisfied for the scaling vector $\alpha$ with $|\alpha|_\infty<1$. Therefore, one may  replace the above theorem with the following, which is immediate from Theorem \ref{prelthm}.
\begin{theorem}
If $|\alpha|_\infty<1$, then  for each $n \in \mathbb{N}$, the (single-valued) fractal operator $\mathcal{F}_{\Delta, B_n}^\alpha: \mathcal{C}(I)\to \mathcal{C}(I)$ is bounded below and not compact.

  \end{theorem}

\subsection*{Acknowledgements}
  The first author thanks the University Grants
  Commission (UGC), India for financial support in the form of a Junior Research Fellowship.


\begin{thebibliography}{1}

\bibitem {Akhtar} M. N. Akhtar, M. G. Prem Prasad, M. A. Navascu\'es, \textit{Box dimensions of $\alpha$-fractal functions}, Fractals 24(3) (2016) 1650037.

\bibitem {Aubin} J. P. Aubin, H. Frankowska, \textit{Set-Valued Analysis}, Birkh\"{a}auser, Boston, 1990.


\bibitem {MF1} M. F. Barnsley, \textit{Fractal functions and interpolation}, Constr. Approx. 2 (1986) 303-329.



\bibitem{BL} P. Bouboulis, L. Dalla, \textit{A general construction of fractal interpolation functions on grids of $R^n$},  European J. Appl.
Math. 18 (2007) 449-476.

\bibitem{CK} A. K. B. Chand, G. P. Kapoor, \textit{Generalized cubic spline fractal interpolation functions}, SIAM J. Numer. Anal. 44 (2)
(2006) 655-676.



    \bibitem{Cheney} E. W. Cheney, \textit{Introduction to Approximation Theory}, Chelsea, New York, 1982.

    \bibitem{PJD} P. J. Davis, \textit{Interpolation and Approximation}, Dover, New York, 2014.
    \bibitem{DS} F. Deustch, I. Singar, \textit{On single-valuedness of convex set-valued maps},  Set-Valued Var Anal. 1 (1993) 97-103.

\bibitem{Hut} J. E. Hutchinson, \textit{Fractals and self-similarity}, Indiana Uni. Math. J.
30(5) (1981) 713-747.

\bibitem{DCL} D-C. Luor, \textit{Fractal interpolation functions with partial self similarity}, J. Math. Anal. Appl. 464 (2018) 911-923.



\bibitem {DVpai} H. N. Mhaskar, D.V. Pai, \textit{Fundamentals of Approximation Theory}, Narosa, 2007.






\bibitem{PRM2} P. R. Massopust, \textit{Interpolation and Approximation with Splines and Fractals}, Oxford University Press, New York, 2010.

\bibitem{PRM1} P. R. Massopust, \textit{Fractal Functions, Fractal Surfaces, and Wavelets}, Academic Press, San Diego, 1994.


\bibitem {MAN1} M. A. Navascu\'es, \textit{Fractal polynomial interpolation}, Z. Anal. Anwend. 25(2) (2005) 401-418.
\bibitem {MAN3} M. A. Navascu\'es, \textit{Fractal trigonometric approximation}, Electron. Trans. Numer. Anal. 20 (2005) 64-74.
\bibitem {MAN2} M. A. Navascu\'es, \textit{Fractal approximation}, Complex Anal. Oper. Theory 4(4) (2010)  953-974.


\bibitem{MAN4}  M. A. Navascu\'es, \textit{Fractal Haar system}, Nonlinear Anal. 74 (2011) 4152-4165.

\bibitem{MAN5}  M. A. Navascu\'es, \textit{Some results on convergence of cubic spline fractal interpolation functions}, Fractals 11(1) (2003) 1-7.

\bibitem{Ri} S. Ri, \textit{A new idea to construct the fractal interpolation function}, Indagat. Math. 29 (2018) 962-971.
\bibitem{S} N. A. Secelean, \textit{Continuous dependence on a parameter of the countable fractal interpolation function},  Carpathian J. Math. 27(2011) 131-141.






\bibitem {AKVarma}  A. K. Varma, \textit{An interpolatory rational trigonometric approximation}, SIAM J. Numer. Anal. 18(5) (1981) 897-899.

\bibitem {Vij1} N. Vijender, \textit{Bernstein fractal trigonometric  approximation}, Acta Appl Math. 159 (2019) 11-27.
\bibitem {Vij2} N. Vijender, \textit{Bernstein fractal rational approximants with no condition on scaling vectors}, Fractals 26(4) (2018) 1850045.
\bibitem {Vij3} N. Vijender, \textit{Fractal perturbation of shaped functions: Convergence independent of scaling}, Mediterr. J. Math. 211 (2018) 1-16.

\bibitem {PV2}  P. Viswanathan, M. A. Navascu\'es, A. K. B. Chand, \textit{Fractal perturbation preserving fundamental shapes: Bounds on
the scale factors}, J. Math. Anal. Appl. 419 (2014) 804-817.
\bibitem {Rational} P. Viswanathan, A. K. B. Chand, \textit{Fractal rational functions and their approximation properties}, J. Approx. Theory 185 (2014) 31-50.
\bibitem{WY} H-Y. Wang, J-S. Yu, \textit{Fractal interpolation functions with variable parameters and their analytical properties}, J. Approx. Theory 175 (2013) 1-18.

\end{thebibliography}
\end{document}